\documentclass{amsart}
\usepackage{amssymb}
\usepackage{mathrsfs}
\usepackage{stmaryrd}
\usepackage{bbm}
\usepackage{oldgerm}
\usepackage[english]{babel}
\usepackage[T1]{fontenc}
\usepackage[latin1]{inputenc}
\usepackage[all]{xy}
\usepackage{mathabx}
\usepackage{graphicx}
\usepackage[colorlinks,linkcolor=red,anchorcolor=red,citecolor=blue]{hyperref}

\newtheorem{theorem}[subsection]{Theorem}
\newtheorem{proposition}[subsection]{Proposition}

\newtheorem{lemma}[subsection]{Lemma}
\newtheorem{conjecture}[subsection]{Conjecture}

\theoremstyle{definition}

\newtheorem{definition}[subsection]{Definition}
\newtheorem{remark}[subsection]{Remark}

\newtheorem{example}[subsection]{Example}
\numberwithin{equation}{subsection}

\newcommand{\ol}{\overline}
\newcommand{\ul}{\underline}
\newcommand{\wt}{\widetilde}

\newcommand{\spec}{\mathrm{Spec}}
\newcommand{\sep}{\mathrm{sep}}
\newcommand{\alg}{\mathrm{alg}}
\newcommand{\gal}{\mathrm{Gal}}
\newcommand{\ch}{\mathrm{char}}
\newcommand{\cl}{\mathrm{cl}}
\newcommand{\fm}{\mathfrak m}
\newcommand{\FE}{\text{F\'E}}
\newcommand{\Hom}{\mathrm{Hom}}
\newcommand{\pr}{\mathrm{pr}}

\newcommand{\gp}{\mathrm{gp}}
\newcommand{\Ext}{\mathrm{Ext}}
\newcommand{\F}{\mathrm{F}}
\newcommand{\wh}{\widehat}

\newcommand{\iso}{\xrightarrow{\sim}}

\newcommand{\sO}{\mathscr{O}}
\newcommand{\bv}{\mathbf v}
\newcommand{\cF}{\mathcal F}
\newcommand{\bZ}{\mathbb Z}
\newcommand{\sF}{\mathscr F}
\newcommand{\bA}{\mathbb A}
\newcommand{\bQ}{\mathbb Q}
\newcommand{\bT}{\mathbb T}
\newcommand{\bG}{\mathbb G}
\newcommand{\rL}{\mathrm L}
\newcommand{\bF}{\mathbb F}
\newcommand{\sG}{\mathscr G}
\newcommand{\rh}{\mathrm{h}}

\begin{document}

\title{Purely inseparable extensions and ramification filtrations}

\author{Haoyu Hu}

\address{Max-Planck-Institut f\"ur Mathematik, Vivatsgasse 7, Bonn 53111, Germany}

\email{huhaoyu1987@163.com, huhaoyu@mpim-bonn.mpg.de}



\begin{abstract}
In this article, we investigate the shift of Abbes and Saito's ramification filtrations of the absolute Galois group of a complete discrete valuation field of positive characteristic under a purely inseparable extension. We also study a functoriality property for characteristic forms.
\end{abstract}

\maketitle
\tableofcontents

\section{Introduction}
\subsection{}\label{K}
In this article, $K$ denotes a discrete valuation field, $\sO_K$ the integer ring of $K$, $F$ the residue field of $K$, $K^{\alg}$ an algebraic closure of $K$, $K^{\sep}$ the separable closure of $K$ in $K^{\alg}$, $\bv_K:K^{\alg}\rightarrow \mathbb Q$ a valuation normalized by $\bv_K(K)=\mathbb Z$ and $G_K$ the Galois group of $K^{\sep}$ over $K$. We assume that the characteristic of the residue field $F$ is $p>0$. 

\subsection{}
Assume that $\sO_K$ is henselian or complete.  When $F$ is perfect,  a classical upper numbering ramification filtration $\{G^r_{K,\cl}\}_{r\in \mathbb Q_{\geq 0}}$ on $G_K$ has been known for a long time \cite{serre cl}. This filtration gives a delicate description of the wild inertia subgroup of $G_K$ and has many applications in arithmetic geometry. For instance, this filtration contributes a local invariant called the {\it Swan conductor} to the Grothendieck-Ogg-Shafarevich formula that computes the Euler-Poincar\'e characteristic of an $\ell$-adic sheaf on a smooth and projective curve over an algebraically closed field of charateristic $p>0$ ($p\neq \ell$) \cite{sga5}. In higher dimensional situations, the ramification phenomena involving purely inseparable residue extension is the main difficulty to generalize the upper numbering ramification filtration. In \cite{as i, as ii}, Abbes and Saito overcame this difficulty using rigid geometry and defined two decreasing filtrations $\{G_K^r\}_{r\in \bQ_{\geq 1}}$ and $\{G_{K,\log}^r\}_{r\in \bQ_{\geq 0}}$ on $G_K$ called the ramification filtration and the logarithmic ramification filtration, respectively.

\subsection{}\label{as12}
In \cite{as i}, many properties for the two ramification filtrations are provided. For any $r\in\bQ_{\geq 0}$, we have $G^{r+1}_{K,\log}\subseteq G^{r+1}_{K}\subseteq G^{r}_{K,\log}$. If $F$ is perfect, for any $r\in\bQ_{\geq 0}$, we have $G^{r+1}_K=G^{r}_{K,\log}=G^{r}_{K,\mathrm{cl}}$.   Let $K'$ be a finite separable extension of $K$ contained in $K^{\sep}$ of ramification index $e$. We identify $G_{K'}=\gal(K^{\sep}/K')$ as a subgroup of $G_K$. Then we have 
$G_{K'}^{er}\subseteq G_K^r$ with an equality when $K'/K$ is unramified, and $G_{K',\log}^{er}\subseteq G_{K,\log}^{r}$ with an equality when $K'/K$ is tamely ramified. 

\subsection{}\label{purinsup}
We further assume that $K$ is equal characteristic. Let $K'$ be a finite and purely inseparable extension of $K$ contained in $K^{\alg}$. The homomorphism $\gamma: G_{K'}=\gal(K^{\alg}/K')\rightarrow G_K$ is an isomorphism and we identify $G_{K'}$ and $G_K$ by $\gamma$. If the residue field  $F$ is perfect, we see that $G^r_{K',\cl}=G^r_{K,\cl}$ for any $r\geq 0$ by considering the lower numbering filtration and the Herbrand function (cf. \cite[Lemma 2.3.1]{tey}).  It is nature to ask: is this equality valid for Abbes and Saito's two ramification filtrations?
If not, what is the shift of the two ramification filtrations
under a purely inseparable extension? In the first half of this article,  we answer these questions by the following theorem, which supplements a property for Abbes and Saito's ramification theory.

\begin{theorem}[{Theorem \ref{thleft} and \ref{thright}}]
\label{pureinsepramintro}
We denote by $e$ the ramification index of $K'/K$ and by $f$ the dual ramification index of $K'/K$ (see \ref{eri}). Then we have the following inclusions: 
\begin{itemize}
\item[(i)]
For any $r\in\bQ_{\geq 1}$, 
\begin{equation*}
G^{er}_{K'}\subseteq G^r_K\ \ \ \textrm{ and}\ \ \ G_K^{fr}\subseteq G^r_{K'}.
\end{equation*}
\item[(ii)]
For any $r\in\bQ_{\geq 0}$, 
\begin{equation*}
G^{er}_{K',\log}\subseteq G^r_{K,\log}\ \ \ \textrm{and}\ \ \ G_{K,\log}^{fr}\subseteq G^r_{K',\log}.
\end{equation*}
\end{itemize}
\end{theorem}
When the residue field $F$ is not perfect, the shifts $e$ and $f$ in the above theorem is optimal in general  (Example \ref{Kn} and \ref{exsamec}). The proof of the theorem relies on a mimic of the shift property for separable extensions as in \ref{as12} and the fact that the Frobenius map induces the identity of absolute Galois groups. 

\subsection{}
In the rest of the introduction, we assume that $K$ is geometric, i.e., there exists a smooth variety $X$ over a perfect field $k$ of characteristic $p>0$, and a generic point  $\xi$ of an irreducible Cartier divisor of $X$ such that $\sO_K\iso\sO^{\mathrm{h}}_{X,\xi}$. For each $r\in \bQ_{\geq 1}$, we put $G^{r+}_K=\overline{\bigcup_{b>r} G^b_K}$. For any $r\in \bQ_{> 1}$, the graded piece $G_K^r/G_K^{r+}$ is abelian and $p$-torsion, and we have an injective homomorphism, called the characteristic form \cite[Proposition 2.28]{wr}
\begin{equation*}
\ch:\Hom_{\bF_p}(G^r_K/G^{r+}_K,\bF_p)\rightarrow \Omega^1_{\sO_K}\otimes_{\sO_K}\overline N_K^{-r}, 
\end{equation*}
where $\overline N^{-r}$ denotes the $1$-dimensional $\overline F$-vector space ($\overline F$ denotes the residue field of $\sO_{K^{\sep}}$)
\begin{equation*}
\overline N_K^{-r}=\{x\in K^{\sep}\,;\, \bv_K(x)\geq -r\}/\{x\in K^{\sep}\,;\, \bv_K(x)> -r\}.
\end{equation*}
Let $K'/K$ be a finite and purely inseparable extension, $e$ the ramification index of $K'/K$ and $f$ the dual ramification index of $K'/K$ (\ref{eri}). By Theorem \ref{pureinsepramintro}, we have canonical maps of graded pieces
\begin{equation}\label{gradpiece}
\gamma:G^{er}_{K'}/G^{er+}_{K'}\rightarrow G^{r}_{K}/G^{r+}_{K},\ \ \ \textrm{and}\ \ \ \gamma^{-1}:G^{fr}_{K}/G^{fr+}_{K}\rightarrow G^{r}_{K'}/G^{r+}_{K'}.
\end{equation}
In the second part of this article, we show a functoriality property for characteristic forms (Theorem \ref{themgr}). Apply it to the extension $K'/K$, we obtain the following proposition that describes images of \eqref{gradpiece}.

\begin{proposition}[{Proposition \ref{coropurins}}]\label{coropurinsintro}
Let $n$ be the integer satisfying $ef=p^n$.
Then, for any $r\in \bQ_{>1}$, we have the following commutative diagrams
\begin{equation*}
\xymatrix{\relax
\Hom_{\bF_p}(G^r_K/G^{r+}_K,\bF_p)\ar[r]^-(0.5){\ch}\ar[d]_{\gamma^{\vee}}&\Omega^1_{\sO_K}\otimes_{\sO_K}\overline N_K^{-r}\ar[d]^{\theta}\\
\Hom_{\bF_p}(G^{er}_{K'}/G^{er+}_{K'},\bF_p)\ar[r]^-(0.5){\ch}&\Omega^1_{\sO_{K'}}\otimes_{\sO_{K'}}\overline N_{K'}^{-er}}
\end{equation*}
\begin{equation*}
\xymatrix{\relax
\Hom_{\bF_p}(G^r_{K'}/G^{r+}_{K'},\bF_p)\ar[r]^-(0.5){\ch}\ar[d]_{(\gamma^{-1})^{\vee}}&\Omega^1_{\sO_{K'}}\otimes_{\sO_{K'}}\overline N_{K'}^{-r}\ar[d]^{\sigma}\\
\Hom_{\bF_p}(G^{fr}_{K}/G^{fr+}_{K},\bF_p)\ar[r]^-(0.5){\ch}&\Omega^1_{\sO_{K}}\otimes_{\sO_{K}}\overline N_{K}^{-fr}}
\end{equation*}
where $\theta$ (resp. $\sigma$) denotes the canonical map of differentials induced by the inclusion $K\subset K'$
(resp. $K'\rightarrow K$, $x'\mapsto x'^{p^n}$).
\end{proposition}

There is an analogue of characteristic forms for graded pieces of the logarithmic ramification filtration, called the {\it refined Swan conductor} (cf. \cite[Corollary 1.25]{saito cc}). We expect a functorial property for the refined Swan conductor similar to Proposition \ref{coropurinsintro}.

\subsection{}
To measure the ramification of an $\ell$-adic sheaf on normal varieties of positive characteristic along a Cartier divisor, Deligne proposed a method by restricting the sheaf to smooth curves. Later, Abbes and Saito's ramification theory allows us to study the ramification of the sheaf at generic points of the divisor. In \cite{ek}, Esnault and Kerz predicted that 
the sharp ramification bound of the sheaf along the divisor by Deligne's approach coincides with that by Abbes and Saito's logarithmic ramification filtrations. In \cite{barr}, the prediction was rigorously formulated and proved for rank $1$ sheaves on smooth varieties. For any sheaves on smooth varieties, the prediction was proved in \cite{hu}.  By Temkin's result \cite{tem}, normal varieties has an inseparable local alteration that kills singularities. We hope that, by applying Temkin's local alteration and the result in this article, we may reduce Esnault and Kerz's prediction for normal varieties to the known smooth situation. 

\subsection{}
This article has two parts. The first part is \S2 --\S 4. After preliminaries on fields, we review Abbes and Saito's ramification filtrations in \S3 and we prove Theorem \ref{pureinsepramintro} in \S4. The second part is \S 5--\S7. We introduce the geometric notation in \S5.  In \S6, we  recall the geometric ramification and characteristic forms for \'etale covers in \cite{wr} and prove the functoriality of characteristic forms. At last, we apply a functoriality property to fields and give a proof of Proposition \ref{coropurinsintro}.

\section{Algebraic prelimilaries}
\subsection{}
In this article, $k$ denotes a field, $k^{\alg}$ an algebraical closure of $k$, $k^{\sep}$ the separable closure of $k$ contained in $k^{\alg}$ and $G_k$ the Galois group of $k^{\sep}$ over $k$.
We denote by $\FE_{/k}$ the category of finite and \'etale $k$-schemes and by $\cF_k$ the functor from $\FE_{/k}$ to the category of finite sets that maps $\spec(l)$ to $\Hom_k(l,k^{\alg})$. The functor $\cF_k$ makes $\FE_{/k}$ a Galois category of group $G_k$. For simplicity, we always put $\cF_k(l)=\cF_k(\spec(l))$ for an object $\spec(l)$ in $\FE_{/k}$.

\begin{lemma}\label{purins}
Assume that the characteristic of $k$ is $p>0$ and let 
 $k'$ be a purely inseparable extension of $k$.
\begin{itemize}
\item[(1)]
Let $l$ be a finite separable extension of $k$. Then, 
the tensor product $k'\otimes_kl$ is a field and, after taking embeddings of $k'$ and $l$ into $k^{\alg}$, we have $k'\otimes_kl\xrightarrow{\sim}k'l$.
\item[(2)]
Let $l'$ be a finite separable extension of $k'$ and we denote by $l$ the separable closure of $k$ in $l'$. Then we have $k'\otimes_kl\iso k'l=l'$.
\item[(3)]
The functor $T:\FE_{/k}\rightarrow\FE_{/k'}$ sending $X$ to $X\otimes_kk'$ is an equivalence of categories.
\item[(4)]
For any finite Galois extension $l/k$, the canonical map $\gal(k'l/k')\rightarrow\gal(l/k)$ is an isomorphism. Taking an embedding of $k'^{\sep}$ in $k^{\alg}$,
we have 
$k^{\sep}\otimes_kk'\iso k'k^{\sep}=k'^{\sep}$. The canonical map $\gamma:G_{k'}\rightarrow G_k$ induced by the inclusion $k\subseteq k'$ is an isomorphism of topological groups.
\end{itemize}
\end{lemma}
\begin{proof}
This proposition is a standard result.  When $[k':k]<+\infty$, (1) and (2) are corollaries of \cite[V, \S 7, Prop 15 and 16]{bour}. In general, $k'$ is a direct limit of finite purely inseparable extensions of $k$ contained in $k'$. By a passage to the limit argument, we can reduce (1) and (2) to the case where $[k':k]<+\infty$. By (1) and (2), for any object $\spec(A')$ of $\FE_{/k'}$, the functor $S: \FE_{/k'}\rightarrow \FE_{/k}$ sending $\spec(A')$ to the spectrum of the \'etale closure of $k$ in $A'$ is a quasi-inverse of $T$. Hence, we obtain (3). Part (4) is deduced by (1), (2) and \cite[V, \S 10, Th\'eor\`em 5]{bour}. 
\end{proof}
 
\begin{lemma}\label{fr id}
Assume that the characteristic of $k$ is $p>0$ and let $F:k\rightarrow k, x\mapsto x^p$ be the Frobenius map. Then, the isomorphism $\gamma:G_k\rightarrow G_{k}$ induced by $F$ is the identity map.
\end{lemma}
\begin{proof}
We decompose $F:k\rightarrow k$ as the isomorphism $F':k\iso k^p,\ \ x\mapsto x^p$ and the canonical inclusion $F'':k^p\rightarrow k$. Let $l$ be a finite Galois extension of $k$. We see that $l^p$ is a finite Galois extension of $k^p$ and $l=kl^p$. We denote by $\ol\gamma':\gal(l^p/k^p)\iso \gal(l/k)$ and $\ol\gamma'':\gal(l/k)\iso \gal(l^p/k^p)$ isomorphisms induced by $F'$ and $F''$, respectively. Let $\sigma$ be an element of $\gal(l/k)$. For any $x\in l$, we have 
\begin{equation*}
(((\ol\gamma'\ol\gamma'')(\sigma))(x))^p=((\ol\gamma'\ol\gamma'')(\sigma))(x^p)=(\ol\gamma''(\sigma))(x^p)=\sigma(x^p)=\sigma(x)^p\in l
\end{equation*}
 i.e., $((\ol\gamma'\ol\gamma'')(\sigma))(x)=\sigma(x)$. Hence, the isomorphism $\ol\gamma'\ol\gamma'':\gal(l/k)\iso\gal(l/k)$ induced by $F$ is the identity. Hence, $\gamma$ is also an identity.
\end{proof}

\subsection{}\label{erd}
Assume that the characteristic of $k$ is $p>0$ and let 
 $k'$ be a purely inseparable extension of $k$. For an integer $m\geq 0$, we denote by $k^{p^{-m}}$ the field $\{x\in k^{\alg}\,|\, x^{p^m}\in k\}$, which is purely inseparable over $k$. 
 We say that $k'/k$ is {\it untwisted} if $k\not\subseteq k'^p$.  We say that $k'/k$ has {\it finite exponent} if there exists an integer $n\geq 0$ such that $k'^{p^n}\subseteq k$. We say  $k'/k$ has {\it exponent} $n\in \mathbb N$ if $n$ is the smallest integer such that $k'^{p^n}\subseteq k$.



\subsection{}\label{eri}
We continue the notation of \ref{K}. For an extension $L$ of $K$   which is a discrete valuation field with finite ramification index, we denote by $\sO_L$ its integer ring, by $\fm_L$ the maximal ideal of $\sO_L$ and by $\pi_L$ a uniformizer of $L$. Assuming that $L$ is contained in $K^{\alg}$, we denote by $\bv_L$ the valuation of $K^\alg$ normalized by $\bv_L(\pi_L)=1$. We define an ultra-metric norm on $K^{\alg}$ by $|x|_L=p^{-\bv_L(x)}$, for any $x\in K^{\alg}$. 
We have $\bv_L=e\bv_K$ and $|\cdot|_L=(|\cdot|_K)^e$, where $e$ denotes the ramification index of $L/K$.
If $\sO_L$ is complete, for a positive integer $n$, we denote by $D_L^n$ the $n$-dimensional closed poly-disc of radius $1$ over $L$. 

Assume that the characteristic of $K$ is $p>0$ and let $K'$ be a discrete valuation field contained in $K^{\alg}$, purely inseparable over $K$. Let $t\geq 0$ be an integer number such that  $K^{p^{-t}}\subseteq K'$ and $K^{p^{-t-1}}\not\subseteq K'$. We call the ramification index of $K'/K^{p^{-t}}$ the {\it strict ramification index} of $K'/K$. Assume that $K'/K$ has exponent $n\geq 1$.  We call the ramification index of $K^{p^{-n}}/K'$ the {\it dual ramification index} of $K'/K$.

\section{Abbes and Saito's ramification filtrations}
\subsection{}
In this section, we assume that $\sO_K$ is complete. let $A$ be a finite and flat $\sO_K$-algebra, $Z=\{z_1,\ldots,z_n\}$ a finite set of elements generating $A$ over $\sO_K$,  
\begin{equation*}
\sO_K[x_1,\ldots,x_n]/I_Z\xrightarrow{\sim} A,\ \ \ x_i\mapsto z_i
\end{equation*}
a presentation of $A$ associated with $Z$ and $\{f_1,\ldots,f_m\}$ a finite set of generators of the ideal $I_Z$. For any rational number $a>0$, we define an affinoid subdomain $X^a_Z$ of $D_K^n$  by
\begin{align*}
X^a_Z&=\{\ul{x}\in D_K^n;\,\bv_K(f(\ul{x}))\geq a\; \text{for all}\; f\in I_Z\}\\
&=\{\ul{x}\in D_K^n;\,\bv_K(f_i(\ul{x}))\geq a\; \text{for all}\; i=1,\cdots, m\}.
\end{align*}
The collection of affinoid domains $\big\{X^a_Z\big\}_{Z}$ form a projective system when $Z$ goes through all finite sets of elements generating $A$ over $\sO_K$. We denote by $\pi_0^{\mathrm{geom}}(X^a_Z)$ the set of the geometric connected components of $X^a_Z$ with respect to either the weak or the strong $G$-topology.  The projective system $\big\{\pi_0^{\mathrm{geom}}(X^a_Z)\big\}_{Z}$ when $Z$ goes through all finite sets of generators of $A$ is constant (\cite[Lemma 3.1]{as i}).  

Let $\spec(L)$ be an object of $\FE_{/K}$ and $A=\sO_L$ the normalization of $\sO_K$ in $L$. We define a functor $\cF^a_K$ from $\FE_{/K}$ to the category of finite sets by sending $\spec(L)$ to $\pi_0^{\mathrm{geom}}(X^a_Z)$. We simply put $\cF_K^a(L)=\cF_K^a(\spec(L))$. The canonical injection 
\begin{equation}\label{f0tofa}
\lambda:\Hom_K(L,K^\alg)\rightarrow X^a_Z(K^\alg),\ \ \ (\phi:L\rightarrow K^{\alg}) \mapsto (\phi(z_1),\ldots,\phi(z_n)).
\end{equation}
induces a surjective map
\begin{equation}\label{FtoFa}
\ol\lambda:\cF_K(L)\rightarrow \cF^a_K(L),
\end{equation}
compatible with the continuous $G_K$-actions. For any rational number $b>a$, the canonical inclusion $X^b_Z\subseteq X^a_Z$ induces a surjective map $\cF^b_K(L)\rightarrow \cF^a_K(L)$
that factors through $\ol\lambda:\cF_K(L)\rightarrow \cF^a_K(L)$. In summary, we have natural transformations $\cF_K\rightarrow \cF_K^a$ and $\cF_K^b\rightarrow \cF_K^a$ for rational numbers $b>a>0$.

\begin{theorem}[{\cite[Theorem 3.3]{as i}}]\label{thmramfil}
(1) For each rational number $a>0$, there exists a unique closed normal subgroup $G_K^a$ of  $G_K$, such that, for any finite separable extension $L$ of $K$ contained in $K^{\sep}$, we have 
\begin{equation*}
\cF_K^a(L)=\cF_K(L)/G^a_K.
\end{equation*}
In particular, assuming $L/K$ is Galois, we have the following canonical isomorphism 
\begin{equation*}
\cF_K^a(L)\xrightarrow{\sim}\gal(L/K)/\gal(L/K)^a
\end{equation*}
of finite sets with continuous $G_K$-actions, where $\gal(L/K)^a$ denotes the image of $G_K^a$ in $\gal(L/K)$.

(2) Let $b>a>0$ be two rational numbers. We have $G^b_K\subseteq G^a_K$.
\end{theorem}

\begin{definition}\label{rambound}
Extending $G^0_K=G_K$, the decreasing filtration $\{G^a_K\}_{a\in \mathbb Q_{\geq 0}}$ in Theorem \ref{thmramfil} is called the {\it ramification filtration} of $G_K$. For a real number $b>0$, we put 
\begin{equation*}
G_K^{b+}=\ol{\bigcup_{r>b}G^r_K}.
\end{equation*}
Let $\spec(L)$ be an object of $\FE_{/K}$. For a rational number $a>0$, we say the ramification of $L/K$ is {\it bounded} by $a$ (resp. $a+$) if $\ol\lambda:\cF_K(L)\rightarrow \cF^a_K(L)$ is  bijective (resp. $\ol\lambda:\cF_K(L)\rightarrow \cF^b_K(L)$ is a bijection for any $b>a$).
Assume that $L$ is a finite Galois extension of $K$. The quotient filtration $\{\gal(L/K)^a\}_{a\in \mathbb Q_{\geq 0}}$ is call the ramification filtration of $\gal(L/K)$. For a rational number $a>0$, the ramification of $L/K$  is bounded by $a$ (resp. $a+$) if $\gal(L/K)^a=\{1\}$ (resp. $\gal(L/K)^{a+}=\{1\}$). We define the {\it conductor} $c$ of $L/K$ as the infimum of rational numbers $r > 0$ such that the ramification of $L/K$ is bounded by $r$. Then $c$ is a rational number and the ramification of $L/K$ is bounded by $c+$ but not bounded by $c$ (\cite[Proposition 6.4]{as i}).

\end{definition}

\begin{proposition}[{\cite[Proposition 3.7]{as i}}]
(1) For any rational number $1\geq a>0$, the group $G^a_K$ is the inertia subgroup of $G_K$ and $G^{1+}_K$ is the wild inertia subgroup of $G_K$.

(2) Let $L$ be a finite separable extension of $K$ contained in $K^\sep$ of ramification index $e$. Then, for any rational number $a\geq 0$, we have $G_L^{ea}\subseteq G^a_K$. If $L/K$ is unramified, the inclusion is an equality.

(3) If the residue field $F$ of $\sO_K$ is perfect, we have $G^{a+1}_K=G^a_{K,\cl}$, where $G^a_{K,\cl}$ denotes the classical upper numbering filtration (\cite{serre cl}).
\end{proposition}

\subsection{}
In the following of this section,  $L$ denotes a finite separable extension of $K$ of ramification index $e$. Let $Z=\{z_1,\ldots,z_n\}$ be a finite set of elements generating $\sO_L$ over $\sO_K$ and we assume that $Z$ contains a uniformizer of $\sO_L$. Let $P$ be a subset of $\{1,\cdots,n\}$ such that $\{z_i\}_{i\in P}$ contains a uniformizer of $\sO_L$ but does not contain $0$.  We call $(Z,P)$ a {\it logarithmic system} of generators of $\sO_L$ over $\sO_K$. For each $i\in P$, we put $\bv_L(z_i)=e_i$. Let 
\begin{equation*}
\sO_K[x_1,\ldots,x_n]/I_Z\xrightarrow{\sim} \sO_L,\ \ \ x_i\mapsto z_i
\end{equation*}
be the representation of $\sO_L$ associated with $Z$ and $\{f_1,\ldots, f_m\}$ a finite set of generators of the ideal $I_Z$. For each $i\in P$, we take a lifting $h_i\in \sO_K[x_1,\ldots,x_n]$ of the unit $u_i=z_i^{e}/\pi_K^{e_i}\in \sO_L$, and, for each pair $(i,j)\in P^2$, we take a lifting $g_{i,j}\in\sO_K[x_1,\ldots,x_n]$ of the unit $v_{i,j}=z_i^{e_j}/z_j^{e_i}\in\sO_L$. We define an affinoid subdomain $Y^a_{Z,P}$ of $D^n$  by
\begin{equation*}
Y^a_{Z,P}=\left\{\ul{x}=(x_1,\ldots,x_n)\in D_K^n\, \Bigg|
\begin{array}{ccc}\bv_K(f_i(\ul{x}))\geq a,\; \text{for}\; 1\leq i\leq m\\
\bv_K\big(x_i^e-\pi_K^{e_i}h_i(\ul x)\big)\geq a+e_i\; \text{for}\; i\in P\\
\bv_K\big(x_i^{e_j}-x_j^{e_i}g_{i,j}(\ul x)\big)\geq a+e_ie_j/e\; \text{for}\; (i,j)\in P^2
\end{array}\right\}
\end{equation*}
We see that $Y^a_{Z,P}\subseteq X^a_Z$.
The affinoid domains $\big\{Y^a_{Z,P}\big\}_{(Z,P)}$ form a projective system when $(Z,P)$ goes through all logarithmic systems of generators. We denote by $\pi_0^{\mathrm{geom}}(Y^a_{Z,P})$ the set of the geometric connected components of $Y^a_{Z,P}$ with respect to either the weak or the strong $G$-topology.  The projective system $\big\{\pi_0^{\mathrm{geom}}(Y^a_{Z,P})\big\}_{(Z,P)}$ when $(Z,P)$ goes through all logarithmic systems of generators is constant (\cite[Lemma 3.10]{as i}).  

We  take $\cF_{K,\log}^a(L)=\pi_0^{\mathrm{geom}}(Y^a_{Z,P})$. Since an object of $\FE_{/K}$ is a finite disjoint union of connected finite \'etale schemes $X=\coprod_i\spec(L_i)$ over $K$, the map $\cF^a_{K,\log}$ is canonically extended to a functor from $\FE_{/K}$ to the category of finite sets by sending $X$ to $\coprod_i\cF_{K,\log}^a(L_i)$. The canonical injection 
\begin{equation}\label{logf0tofa}
\mu:\Hom_K(L,K^\alg)\rightarrow Y^a_{Z,P}(K^\alg),\ \ \ (\phi:L\rightarrow K^{\alg}) \mapsto (\phi(z_1),\ldots,\phi(z_n)).
\end{equation}
induces a surjective map
\begin{equation}\label{logFtoFa}
\ol\mu:\cF_K(L)\rightarrow \cF^a_{K,\log}(L)=\pi_0^{\mathrm{geom}}(Y^a_{Z,P}),
\end{equation}
compatible with the continuous $G_K$-actions. For any rational number $b>a$, the canonical inclusion $Y^b_{Z,P}\subseteq Y^a_{Z,P}$ induces a surjective map $\cF^b_{K,\log}(L)\rightarrow \cF^a_{K,\log}(L)$
that factors through $\ol\mu:\cF_{K}(X)\rightarrow \cF^a_{K,\log}(X)$. In summary, we have natural transformations $\cF_{K}\rightarrow \cF_{K,\log}^a$ and $\cF_{K,\log}^b\rightarrow \cF_{K,\log}^a$ for rational numbers $b>a>0$.

\begin{theorem}[{\cite[Theorem 3.11]{as i}}]\label{thmlogramfil}
(1) For each rational number $a>0$, there exists a unique closed normal subgroup $G_{K,\log}^a$ of  $G_K$, such that, for any finite extension $L$ contained in $K^{\sep}$, we have 
\begin{equation*}
\cF_{K,\log}^a(L)=\cF_K(L)/G^a_{K,\log}.
\end{equation*}
In particular, assuming that $L/K$ is Galois, we have the following canonical isomorphism 
\begin{equation*}
\cF_{K,\log}^a(L)\xrightarrow{\sim}\gal(L/K)/\gal(L/K)^a_{\log}
\end{equation*}
of finite sets with continuous $G_K$-actions, where $\gal(L/K)^a_{\log}$ denotes the image of $G_{K,\log}^a$ in $\gal(L/K)$.

(2) Let $b>a>0$ be two rational numbers. We have $G^b_{K,\log}\subseteq G^a_{K,\log}$.
\end{theorem}

\begin{definition}
Extending $G^0_{K,\log}$ by the inertia subgroup $I_K$ of $G_K$, the decreasing filtration $\{G^a_{K,\log}\}_{a\in \mathbb Q_{\geq 0}}$ in Theorem \ref{thmramfil} is called the {\it logarithmic ramification filtration} of $G_K$. For a real number $b>0$, we put 
\begin{equation*}
G_{K,\log}^{b+}=\ol{\bigcup_{r>b}G^r_{K,\log}}.
\end{equation*}
 For a rational number $a>0$, we say the logarithmic ramification of $L/K$ is {\it bounded} by $a$ (resp. $a+$) if $\ol\mu:\cF_K(L)\rightarrow \cF^a_{K,\log}(L)$ is a bijection (resp. $\ol\mu:\cF_K(L)\rightarrow \cF^b_{K,\log}(L)$ is bijective for any $b>a$).
Assume that $L/K$ is Galois. The quotient filtration $\{\gal(L/K)^a_{\log}\}_{a\in \mathbb Q_{\geq 0}}$ is call the logarithmic ramification filtration of $\gal(L/K)$. For a rational number $a>0$, the logarithmic ramification of $L/K$ {\it is  bounded by} $a$ (resp. $a+$) if $\gal(L/K)^a_{\log}=\{1\}$ (resp. $\gal(L/K)^{a+}_{\log}=\{1\}$). We define the {\it logarithmic conductor} $c$ of $L/K$ as the infimum of rational numbers $r > 0$ such that the logarithmic ramification of $L/K$ is bounded by $r$. Then $c$ is a rational number and the logarithmic ramification of $L/K$ is bounded by $c+$ but not bounded by $c$ ([AS1, 9.5]).
\end{definition}

\begin{proposition}[{\cite[Proposition 3.15]{as i}}]
(1) The subgroup $G^{0+}_{K,\log}$ is the wild inertia subgroup of $G_K$.

(2) Let $L$ be a finite separable extension of $K$ contained in $K^\sep$ of ramification index $e$. Then, for any rational number $a\geq 0$, we have $G_{L,\log}^{ea}\subseteq G^a_{K,\log}$. If $L/K$ is tamely ramified, the inclusion is an equality.

(3) For a rational number $a\geq 0$, we have $G^{a+1}_K\subseteq G^a_{K,\log}\subseteq G^a_{K}$.  If the residue field of $K$ is perfect, we have $G^{a}_{K,\log}=G^a_{K,\cl}$.
\end{proposition}

\begin{remark}
Let $K_0$ be a henselian discrete valuation field and we denote by $\wh K_0$ the fraction field of the formal completion $\wh\sO_{K_0}$. By  Artin's algebraization theorem (\cite{artin}), we have a canonical isomorphism 
$G_{\wh K_0}\iso G_{K_0}$. Hence Abbes and Saito's ramification filtrations are also applied to henselian discrete valuation fields.
\end{remark}

\section{Shift of ramification filtrations by purely inseparable extensions}

\subsection{}
In this section, we assume that $\sO_K$ is complete and that the characteristic of $K$ is $p>0$. Let $K'$ be a complete discrete valuation field purely inseparable over $K$ contained in $K^{\alg}$, $m$ the ramification index of $K'/K$ and $\gamma: G_{K'}\iso G_K$ the isomorphism induced by the inclusion $K\subseteq K'$ (Lemma \ref{purins}).

\begin{lemma}\label{diagaffinoid}
 Let $L$ be a finite separable extension of $K$ contained in $K^{\sep}$ and $L'=K'L$. We take a set of generators $Z=\{z_1,\ldots,z_n\}$ of $\sO_L$ over $\sO_K$ and we extend to a set of generators $Z=\{z_1,\ldots,z_n,z_{n+1},\ldots, z_{n'}\}$ of $\sO_{L'}$ over $\sO_{K'}$.We have the following presentations
\begin{align*}
&\sO_K[x_1,\ldots, x_n]\big/I_Z\iso\sO_L, \ \ \ x_i\mapsto z_i.\\
&\sO_{K'}[x_1,\ldots, x_{n'}]\big/I_{Z'}\iso\sO_{L'}, \ \ \ x_i\mapsto z_i.
\end{align*}
We take a logarithmic systems of generators $(Z,P)$ of $\sO_L$ (resp. $(Z', P')$ of $\sO_{L'}$) and we assume that $P\subseteq P'$.

Then,
for any rational number $a\geq 0$, we have morphisms of affinoids $\pr_X:X^{ma}_{Z'}\to X^a_Z\times_KK'$ and $\pr_Y:Y^{ma}_{Z',P'}\to Y^a_{Z,P}\times_KK'$ such that the following diagrams are commutative
\begin{equation}\label{keydiag}
\xymatrix{\relax
\Hom_{K'}(L',K^{\alg})\ar[r]^-(0.5){\lambda'}\ar[d]_{\cdot|_L}&X^{ma}_{Z'}(K^{\alg})\ar[d]^{\pr_X}&&\Hom_{K'}(L',K^{\alg})\ar[r]^-(0.5){\mu'}\ar[d]_{\cdot|_L}&Y^{ma}_{Z',P'}(K^{\alg})\ar[d]^{\pr_Y}\\
\Hom_{K}(L,K^{\alg})\ar[r]_-(0.5){\lambda}&X^a_Z({K^{\alg}})&&\Hom_{K}(L,K^{\alg})\ar[r]_-(0.5){\mu}&Y^a_{Z,P}({K^{\alg}})}
\end{equation}
where $\lambda$, $\lambda'$, $\mu$ and $\mu'$ are canonical injections (\ref{f0tofa} and \ref{logf0tofa}).
\end{lemma}

\begin{proof}
It is sufficient to show that $X^{ma}_{Z'}(K^{\alg})$ maps to
$X^{a}_{Z}(K^{\alg})$ (resp. $Y^{ma}_{Z',P'}(K^{\alg})$ maps to $Y^{a}_{Z,P}(K^{\alg})$) through the projection 
\begin{equation*}
\pr_{1\ldots n}:D_{K'}^{n'}(K^{\alg})\rightarrow D^n_K(K^{\alg}),\ \ \  (t_1,\ldots,t_{n'})\mapsto (t_1,\ldots,t_n).
\end{equation*}
 This is easy to verify for $X^{ma}_{Z'}(K^{\alg})$ from the definition. The proof for $Y^{ma}_{Z',P'}(K^{\alg})$ is an analogue of \cite[Lemma 9.6]{as i}. 
 
Let $e$ (resp. $e'$ and $m'$) be the ramification index of $L/K$ (resp. $L'/K'$ and $L'/L$). We have $em'=e'm$. For $i\in P'$, we put $e'_i=\bv_{L'}(z_i)$ and, for $i\in P$, we put $e_i=\bv_{L}(z_i)$. We have $e'_i=m'e_i$ for $i\in P$. We fix $\iota\in P'$ such that $z_\iota$ is a uniformizer of $L'$. Let  $u$ be the unit $\pi_{K'}^m/\pi_K\in \sO_{K'}$, $g\in \sO_{K'}[x_1,\ldots, x_{n'}]$ a lifting of $z^{e'}_\iota/\pi_{K'}\in \sO_{L'}$,  for $i\in P$, $h_i\in \sO_{K'}[x_1,\ldots, x_{n'}]$ a lifting of $z_i/z_\iota^{e'_i}$, $l_i\in \sO_{K'}[x_1,\ldots, x_{n'}]$ a lifting of $z_\iota^{e'_i}/z_i$, $r_i\in \sO_K[x_1,\ldots,x_n]$ a lifting of $z_i^e/\pi_K^{e_i}$, and, for $(i,j)\in P$, $r_{i,j}\in \sO_K[x_1,\ldots,x_n]$ a lifting of $z_j^{e_i}/z_i^{e_j}\in \sO_L$. Let $\ul t'=(t_1,\ldots,t_{n'})\in D_{K'}^{n'}(K^{\alg})$ be an element in $Y^{ma}_{Z',P'}(K^{\alg})$ and we denote by $\ul t=(t_1,\ldots,t_{n})$ its image in  $D^n_K(K^{\alg})$.  Notice that $g(\ul t')$, $h_i(\ul t')$, $l_i(\ul t')$, $r_i(\ul t)$ and $r_{i,j}(\ul t)$  are units in the integer ring $\sO_{K^{\alg}}$ and that $\bv_{K}(t_i)=\bv_{K}(z_i)=e_i/e$ (\cite[Lemma 3.9]{as i}). 

Since $\pr_{1\ldots n}$ maps $X^{ma}_{Z'}(K^{\alg})$ to
$X^{a}_{Z}(K^{\alg})$, to show $\ul t$ is contained in $Y^{a}_{Z, P}(K^{\alg})$, it is sufficient to check that, for $i\in P$, we have $\bv_K\big(\big(t_i^e/\pi_K^{e_i}\big)-r_i(\ul t)\big)\geq a$ and, for $(i,j)\in P^2$, we have $\bv_K\big(\big(t_j^{e_i}/t_i^{e_j}\big)-r_{i,j}(\ul t)\big)\geq a$. Take an $i\in P$. Notice that
\begin{equation*}
z^e_i/\pi_K^{e_i}=\big(\pi_{K'}^m/\pi_K\big)^{e_i}\big(z_\iota^{e'}/\pi_{K'}\big)^{me_i}\big(z_i/z_\iota^{e'_i}\big)^e=u^{e_i}\big(z_\iota^{e'}/\pi_{K'}\big)^{me_i}\big(z_i/z_\iota^{e'_i}\big)^e\in \sO_{L'}.
\end{equation*}
Hence $u^{e_i}g^{me_i}h_i^e\in \sO_{K'}[x_1,\ldots,x_{n'}]$ is another lifting of $z^e_i/\pi_K^{e_i}$. Since $\ul t'$ is contained in $X^{ma}_{Z'}(K^{\alg})$, we have
\begin{equation*}
\bv_{K'}\big(u^{e_i}g^{me_i}(\ul t')h_i^e(\ul t')-r_i(\ul t)\big)\geq ma.
\end{equation*}
Notice that $t_i^e/\pi_K^{e_i}=u^{e_i}\big(t_\iota^{e'}/\pi_{K'}\big)^{me_i}\big(t_i/t_\iota^{e'_i}\big)^e$. To show $\bv_K\big(\big(t_i^e/\pi_K^{e_i}\big)-r_i(\ul t)\big)\geq a$, it is sufficient to check that 
\begin{align*}
\bv_{K'}\left(u^{e_i}\big(t_\iota^{e'}/\pi_{K'}\big)^{me_i}\big(t_i/t_\iota^{e'_i}\big)^e-u^{e_i}g^{me_i}(\ul t')h_i^e(\ul t')\right)\geq ma.
\end{align*}
This is deduced from following inequalities (the definition of $Y^{ma}_{Z',P'}$):
\begin{align*}
&\bv_{K'}\left(\big(t_\iota^{e'}/\pi_{K'}\big)^{me_i}-g^{me_i}(\ul t')\right)\geq \bv_{K'}\left(\big(t_\iota^{e'}/\pi_{K'}\big)-g(\ul t')\right)\geq ma,\\
&\bv_{K'}\left(\big(t_i/t_\iota^{e'_i}\big)^e-h_i^e(\ul t')\right)\geq \bv_{K'}\left(\big(t_i/t_\iota^{e'_i}\big)-h_i(\ul t')\right)\geq ma.
\end{align*}
Take a pair $(i,j)\in P$. Notice that
\begin{equation*}
z_j^{e_i}/z_i^{e_j}=\big(z_j/z_\iota^{e'_j}\big)^{e_i}\big(z_\iota^{e'_i}/z_i\big)^{e_j}.
\end{equation*}
Hence $h_j^{e_i}l_i^{e_j}\in\sO_{K'}[x_1,\ldots,x_{n'}]$ is another lifting of $z_j^{e^i}/z_i^{e_j}$. We have 
\begin{equation*}
\bv_{K'}\big(h_j^{e_i}(\ul t')l_i^{e_j}(\ul t')-r_{i,j}(\ul t)\big)\geq ma.
\end{equation*}
To show $\bv_K\big(\big(t_j^{e_i}/t_i^{e_j}\big)-r_{i,j}(\ul t)\big)\geq a$, it is sufficient to check that 
\begin{equation*}
\bv_K\left(  \big(t_j/t_\iota^{e'_j}\big)^{e_i}  \big(t_\iota^{e'_i}/t_i\big)^{e_j} - h_j^{e_i}(\ul t')l_i^{e_j}(\ul t') \right)\geq a.
\end{equation*}
This is deduced from the following inequalities (the definition of $Y^{ma}_{Z',P'}$):
\begin{align*}
&\bv_{K'}\left(\big(t_j/t_\iota^{e'_j}\big)^{e_i}-h_j^{e_i}(\ul t')\right)\geq \bv_{K'}\left(\big(t_j/t_\iota^{e'_j}\big)-h_j(\ul t')\right)\geq ma,\\
&\bv_{K'}\left( \big(t_\iota^{e'_i}/t_i\big)^{e_j} -    l_i^{e_j}(\ul t')   \right)\geq \bv_{K'}\left( \big(t_\iota^{e'_i}/t_i\big) -    l_i(\ul t')   \right)\geq ma.
\end{align*}
\end{proof}

\begin{lemma}\label{2inclu}
Let $a$ and $b$ be two positive rational numbers.
\begin{itemize}
\item[(1)]
Suppose that, for any finite Galois extension $L/K$ whose ramification is bounded by $a$, the ramification of $LK'/K'$ is bounded by $b$. Then, we have $\gamma(G^{b}_{K'})\subseteq G^a_K$.
\item[(1')]
Suppose that, for any finite Galois extension $L/K$ whose logarithmic ramification is bounded by $a$, the logarithmic ramification of $LK'/K'$ is bounded by $b$. Then, we have $\gamma(G^{b}_{K',\log})\subseteq G^a_{K,\log}$.
\item[(2)]
Suppose that, for any finite Galois extension $L/K$ such that the ramification of $LK'/K'$ is bounded by $b$, the ramification of  $L/K$ is bounded by $a$. Then we have $G^a_K\subseteq \gamma(G^{b}_{K'})$.
\item[(2')]
Suppose that, for any finite Galois extension $L/K$ such that the logarithmic ramification of $LK'/K'$ is bounded by $b$, the logarithmic ramification of  $L/K$ is bounded by $a$. Then we have $G^a_K\subseteq \gamma(G^{b}_{K'})$.
\end{itemize}
\end{lemma}
\begin{proof}
The proofs for the ramification filtration and for the logarithmic ramification filtration are the same. We only verify (1) and (2).

(1) To prove $\gamma(G^{b}_{K'})\subseteq G^a_K$, it is sufficient to show that, for any finite Galois extension $L$ of $K$ contained in $K^{\alg}$, we have $\ol \gamma(\gal(LK'/K')^{b})\subseteq \gal(L/K)^a$, where $\ol\gamma:\gal(LK'/K')\iso \gal(L/K)$ denotes the quotient of $\gamma:G_{K'}\rightarrow G_K$. We denote by $L_0$ the $\gal(L/K)^a$-invariant subfield of $L$. We have 
\begin{equation*}
\gal(L_0/K)^a=(\gal(L/K)/\gal(L/K)^a)^a=\{1\}.
\end{equation*}
By the assumption of (1), we have $\gal(L_0K'/K')^{b}=\{1\}$. Since $\gal(L_0K'/K')^{b}$ is the image of $\gal(LK'/K')^{b}$ in $\gal(L_0K'/K')$, we have $\gal(LK'/K')^{b}\subseteq \gal(LK'/L_0K')$. Hence, we get
\begin{equation*}
\ol\gamma(\gal(LK'/K')^{b})\subseteq \ol\gamma(\gal(LK'/L_0K'))=\gal(L/L_0)=\gal(L_0/K)^a.
\end{equation*}

(2) To prove $G^a_K\subseteq \gamma(G^{b}_{K'})$, it is sufficient to show that, for any finite Galois extension $L$ of $K$ contained in $K^{\alg}$, we have $\gal(L/K)^a\subseteq \ol \gamma(\gal(LK'/K')^{b})$. We denote by $L_0$ the $\ol \gamma(\gal(LK'/K')^{b})$-invariant subfield of $L$. Hence, $L_0K'$ is the $\gal(LK'/K')^{b}$-invariant subfield of $LK'$.
We have
\begin{equation*}
\gal(L_0K'/K')^{b}=(\gal(LK'/K')/\gal(LK'/K')^{b})^{b}=\{1\}.
\end{equation*}
By the assumption of (2), we have $\gal(L_0/K)^a=\{1\}$. Since $\gal(L_0/K)^a=\{1\}$ is the image of $\gal(L/K)^a$ in $\gal(L_0/K)$, we get
\begin{equation*}
\gal(L/K)^a\subseteq \gal(L/L_0)=\ol \gamma(\gal(LK'/K')^{b}).
\end{equation*}
\end{proof}

\begin{proposition}\label{rightinclu}
For any rational number $a\geq 0$, we have 
\begin{equation*}
\gamma(G^{ma}_{K'})\subseteq G^a_K\ \ \ \text{and}\ \ \ \gamma(G^{ma}_{K',\log})\subseteq G^a_{K,\log}.
\end{equation*}
\end{proposition}

\begin{proof}
  We fix a rational number $a\geq 0$. Let $L$ be a finite Galois extension of $K$ contained in $K^{\sep}$ and we take the notation of Lemma \ref{diagaffinoid}. By \eqref{keydiag}, we have the following commutative diagram
\begin{equation*}
\xymatrix{\relax
\cF_{K'}(L')\ar[r]^-(0.5){\ol\lambda'}\ar[d]_{\cdot|_L}&\cF^{ma}_{K'}(L')\ar[d]^{\ol\pr_X}\\
\cF_K(L)\ar[r]_-(0.5){\ol\lambda}&\cF^a_K(L)}
\end{equation*}
where $\cdot|_L$ is a bijection, $\ol\lambda$ and $\ol\lambda'$ are surjections induced by $\lambda$ and $\lambda'$ \eqref{FtoFa} and $\ol\pr_X$ is induced by $\pr_X$ \eqref{keydiag}. By chasing the diagram, $\ol\lambda'$ is a bijection if $\ol\lambda$ is a bijection. By Theorem \ref{thmramfil}, $\ol\lambda$ (resp. $\ol\lambda'$) is an bijection if and only if $\gal(L/K)^a=\{1\}$ (resp. $\gal(L'/K')^{ma}=\{1\}$). Hence, $\gal(L/K)^a=\{1\}$  implies $\gal(L'/K')^{ma}=\{1\}$. By Lemma \ref{2inclu} (1), we get $\gamma(G^{ma}_{K'})\subseteq G^a_K$.

The proof for the logarithmic ramification filtration is the same as above. 
\end{proof}

\begin{lemma}\label{frob fileq}
Let $n>0$ be an integer and we denote by $\gamma_n:G_{K^{p^{-n}}}\iso G_K$ the isomorphism of Galois groups induced by the canonical inclusion $\iota_n:K\rightarrow K^{p^{-n}}$ (Proposition \ref{purins}). Then, for any rational number $a\geq 0$, we have 
\begin{equation*}
\gamma_n\big(G^a_{K^{p^{-n}}}\big)=G^a_K\ \ \ \text{and}\ \ \ \gamma_n\big(G^a_{K^{p^{-n}},\log}\big)=G^a_{K,\log}.
\end{equation*}
\end{lemma}
\begin{proof}
Fix  a rational number $a\geq 0$. We put $\delta_n:K^{p^{-n}}\rightarrow K, x\mapsto x^{p^n}$ and denote by $\psi_n:G_K\iso G_{K^{p^{-n}}}$ the induced isomorphism. Since $\delta_n$ is an isomorphism, we have 
$G^a_K=\psi_n^{-1}(G^a_{K^{p^{-n}}})$. Since the composition $\delta_n$ and $\iota_n$ is the $n$-th power of the Frobenius map, 
 the composition $\gamma_n\psi_n:G_K\iso G_K$ is the identity (Lemma \ref{fr id}). Hence, for any rational number $a>0$, we have 
 \begin{equation*}
 \gamma_n\big(G_{K^{p^{-n}}}^a\big)=(\gamma_n\psi_n)\big(\psi_n^{-1}\big(G_{K^{p^{-n}}}^a\big)\big)=G^a_K.
 \end{equation*}
  The proof for the logarithmic ramification subgroups is the same as above.
\end{proof}

\begin{remark}
When the residue field $F$ of $\sO_K$ is perfect, $K$ is isomorphic to the field of Laurent power series $F((x))$. For any finite purely inseparable extension $K'$ over $K$, we have a canonical isomorphism $K'\iso F((x^{p^{-n}}))=K^{p^{-n}}$, where $n=\log_p[K':K]$. Lemma \ref{frob fileq} implies that $\gamma(G^a_{K',\cl})=G^a_{K,\cl}$ for rational number $a\geq 0$, where $\{G^a_{K,\cl}\}_{a\in \mathbb Q_{\geq 0}}$ denotes the classical upper numbering filtration of $G_K$.  This is a proof of the equality $\gamma(G^a_{K',\cl})=G^a_{K,\cl}$ mentioned in \ref{purinsup} without involving lower numbering filtration.
\end{remark}

\begin{theorem}\label{thright}
Assume that $K'/K$ has strict ramification index $s$ (\ref{eri}). Then, for any rational number $a\geq 0$, we have 
\begin{equation}\label{righteri}
\gamma\big(G^{sa}_{K'}\big)\subseteq G^a_K\ \ \ \text{and}\ \ \ \gamma\big(G^{sa}_{K',\log}\big)\subseteq G^a_{K,\log}.
\end{equation}
\end{theorem}
\begin{proof}
Let $t\geq 0$ be an integer number such that $K^{p^{-t}}\subseteq K'$ and $K^{p^{-t-1}}\not\subseteq K'$. Then, $s$ is the ramification index of $K'/K^{p^{-t}}$. Applying Proposition \ref{rightinclu} to $ K^{p^{-t}}\subseteq K'$ and Lemma \ref{frob fileq} to $K\subseteq K^{p^{-t}}$, we obtain \eqref{righteri}.
\end{proof}

\begin{theorem}\label{thleft}
Assume that $K'/K$ has finite exponent and let $n$ be the dual ramification index of $K'/K$ (\ref{eri}). Then, for any rational number $a\geq 0$, we have 
\begin{equation*}
G^{na}_K\subseteq \gamma\big(G_{K'}^a\big)\ \ \ \textrm{and}\ \ \ G^{na}_{K,\log}\subseteq \gamma\big(G_{K',\log}^a\big).
\end{equation*}
\end{theorem}
\begin{proof}
 We denote by $\theta:G_{K^{p^{-t}}}\iso G_{K'}$ the isomorphism induced by the inclusion $K'\subseteq K^{p^{-t}}$, where $t$ is the exponent of $K'/K$. The ramification index of $K^{p^{-t}}/K'$ is $n$. Fix a rational number $a\geq 0$.
  Applying Proposition \ref{rightinclu} to $\theta$, we obtain $ \theta\big(G_{K^{p^{-t}}}^{na}\big)\subseteq G^a_{K'}$.
Let $\gamma_n:G_{K^{p^{-n}}}\iso G_K$ be the isomorphism induced by the inclusion $K\subseteq K^{p^{-n}}$ as in Lemma \ref{frob fileq}.  Notice that $\gamma_n:G_{K^{p^{-n}}}\rightarrow G_K$ is the composition of $\theta$ and $\gamma$.  By Lemma \ref{frob fileq}, we have
\begin{equation*}
G^{na}_K=(\gamma\theta)\big(G^{na}_{K^{p^{-t}}}\big)\subseteq \gamma\big(G^a_{K'}\big).
\end{equation*}
 The proof for the logarithmic ramification subgroups is the same.
\end{proof}

\begin{remark}
Keep the assumption of Theorem \ref{thleft}. Assume that $f=[K':K]<+\infty$.  The dual ramification index of $K'/K$ is less than $f$. Hence, by Theorem \ref{thleft}, we have 
\begin{equation}\label{extdegleft}
G^{fa}_K\subseteq \gamma\big(G_{K'}^a\big)\ \ \ \textrm{and}\ \ \ G^{fa}_{K,\log}\subseteq \gamma\big(G_{K',\log}^a\big).
\end{equation}
  Let $L$ be a finite separable extension of $K$ contained in $K^{\sep}$ of ramification index $m$. We have a canonical inclusion $G_L\subseteq G_K$. By \cite[Proposition 3.7 and 3.15]{as i}, we have $G^{ma}_L\subseteq G^a_K$ and $G^{ma}_{L,\log}\subseteq G^a_{K,\log}$ for any rational number $a\geq 0$, which are similar to Theorem \ref{thright}. However, we do not have inclusions \eqref{extdegleft} for $L/K$, since the number $[L:K]$ cannot bound the ramification and the logarithmic ramification of $L/K$.
\end{remark}

\subsection{}\label{exnotation}
In the following of this section, let $k$ be an algebraically closed field of characteristic $p\geq 3$, $X=\spec(k[x,y])$ a $2$-dimensional affine space over $k$, $D$ the Cartier divisor $y=0$ of $X$, $\xi$ the generic point of $D$ and $K=\mathrm{Frac}(\widehat\sO_{X,\xi})$. We have $K=k(x)((y))$.

Let 
\begin{equation*}
\widetilde U=\spec(\sO_U[t]/(t^p-t-x/y^{p^{n+1}})),\ \ \ (n\geq 1),
\end{equation*}
 be an Artin-Schreier cover of $U=X-D$ and 
$L/K$ the finite Galois extension associate to the cover $\widetilde U/U$. We have $\gal(\widetilde U/U)=\gal(L/K)\cong \bZ/p\bZ$. Let $\Lambda$ be a finite field of characteristic $\ell\neq p$ and $\sF$ a locally constant and constructible \'etale sheaf of $\Lambda$-modules of rank $1$ associated to a non-trivial character $\psi_0:\gal(\widetilde U/U)\rightarrow \Lambda^{\times}$.

\begin{example}\label{Kn}
 Let 
\begin{equation*}
X_n=\spec(\sO_X[u]/(u^{p^n}-x))=\spec(k[u,y])
\end{equation*}
 be a radicial cover of $X$ and $K_n=K[u]/(u^{p^n}-x)$ the purely inseparable extension of $K$ associated to the cover $X_n/X$. We put $L_n=LK_n$ and we have 
 \begin{equation*}
 L_n=K_n[t]/(t^p-t-u^{p^n}/y^{p^{n+1}})= K_1[t']/(t'^p-t'-u/y^p),
 \end{equation*}
 where $t'=t-\sum^{n-1}_{i=0} (u/y^p)^{p^i}$. We denote by $\sF_n$ the pull back of $\sF$ on $U_n=X_n\times_XU$. The total dimension of $\sF|_K$ (resp. $\sF_n|_{K_n}$) along the divisor $D$ (resp. $D_n=D\times_XX_n$) equals the conductor of $L/K$ (resp. $L_n/K_n$) and the Swan conductor of $\sF|_{K}$ (resp. $\sF_n|_{K_n}$)  equals the logarithmic conductor of $L/K$ (resp. $L_n/K_n$). On the other hand, by \cite[Corollary 3.9.1]{wr} and \cite[Theorem 5.1]{barr}, the total dimensions and the Swan conductors of $\sF|_K$ (resp. $\sF_n|_{K_n}$) can be computed by pull-back to curves. Then, we obtain that 
  \begin{align}\label{dimtotswF}
  c(L/K)=c_{\log}(L/K)=p^{n+1},\ \ \ c(L_n/K_n)=c_{\log}(L_n/K_n)=p,
  \end{align}
where $c$ and $c_{\log}$ denotes the conductor and the logarithmic conductor, respectively. Notice that $[K_n:K]=p^n$ and that $K_n/K$  has exponent $n$ and ramification index $1$. We claim that there is no $\epsilon>0$ such that, for any rational number $a\geq 0$, $G_{K}^{(p^n-\epsilon)a}\subseteq \gamma(G_{K_n}^a)$ and $G_{K,\log}^{(p^n-\epsilon)a}\subseteq \gamma(G_{K_n,\log}^{a})$. Indeed, taking $a=p+\epsilon/p^{n-1}$, we get 
\begin{align*}
\gal(L/K)^{p^{n+1}-\epsilon^2/p^{n-1}}&=\gal(L/K)^{p^{n+1}-\epsilon^2/p^{n-1}}_{\log}\cong \mathbb Z/p\mathbb Z,\\
\gal(L_n/K_n)^{p+\epsilon/p^{n-1}}&=\gal(L_n/K_n)^{p+\epsilon/p^{n-1}}_{\log}\cong \{1\},
\end{align*}
 which imply $G_{K}^{p^{n+1}-\epsilon^2/p^{n-1}}\not\subseteq \gamma(G_{K_n}^{p+\epsilon/p^{n-1}})$ and $G_{K,\log}^{p^{n+1}-\epsilon^2/p^{n-1}}\not\subseteq \gamma(G_{K_n,\log}^{p+\epsilon/p^{n-1}})$. 
The assertion implies that Theorem \ref{thleft} is optimal in general.
\end{example}

\begin{example}\label{exsamec}
We take $n=0$ in \ref{exnotation} and  let 
\begin{equation*}
X'=\spec(\sO_X[v]/(v^p-y))=\spec(k[x,v])
\end{equation*}
 be a radicial cover of $X$ and  $K'=K[v]/(v^p-y)$ the purely inseparable extension of $K$ associated to $X'/X$. We put $L'=LK'$ and we have 
 \begin{equation*}
L'=K'[t]/(t^p-t-x/v^{p^2}),
 \end{equation*}
 We denote by $\sF'$ the pull back of $\sF$ on  $U'=U\times_XX'$. The total dimension of $\sF'|_{K'}$ along the divisor $D'=D\times_XX'$ equals the conductor of $L'/K'$ and the Swan conductor of $\sF'|_{K'}$  equals the logarithmic conductor of $L'/K'$. By \cite[Corollary 3.9.1]{wr} and \cite[Theorem 5.1]{barr} again, we compute the total dimension and the Swan conductors of $\sF'|_{K'}$ by pull-back to curves. We obtain that 
  \begin{align*}
c(L/K)=c_{\log}(L/K)=p,\ \   c(L'/K')=c_{\log}(L'/K')=p^2.
  \end{align*}
Notice that the extension degrees of $K'/K$ is $p$ and that $K'/K$ has exponent $1$ and ramification index $p$.  We claim that there is no $\epsilon>0$, such that, for any rational number $a\geq 0$, $G_{K'}^{(p-\epsilon)a}\subseteq G^a_K$ and $G_{K',\log}^{(p-\epsilon)a}\subseteq G^a_{K,\log}$. Indeed, taking $a=p+\epsilon$, we have 
\begin{equation*}
\gal(L'/K')^{p^2-\epsilon^2}=\gal(L'/K')^{p^2-\epsilon^2}_{\log}\cong\mathbb Z/p\bZ\ \ \ \textrm{and}\ \ \ \gal(L/K)^{p+\epsilon}=\gal(L/K)^{p+\epsilon}_{\log}\cong\mathbb \{1\},
\end{equation*}
which imply $\gamma(G_{K'}^{p^2-\epsilon^2})\not\subseteq G_{K}^{p+\epsilon}$ and $\gamma(G_{K',\log}^{p^2-\epsilon^2})\not\subseteq G_{K,\log}^{p+\epsilon}$. This assertion implies that Theorem \ref{thright} is optimal in general.
\end{example}

\begin{remark}
Teyssier proposed a question \cite[Question 1]{tey} as follows.  Let $g:Y\rightarrow Z$ be a $k$-morphism of irreducible $k$-schemes of finite type and $\sG$ a bounded complex of constructible \'etale sheaves of $\Lambda$-modules on $Y$ with $\mathrm{Supp}(\sG)=Y$. Can we bound the wild ramification of $\mathrm{R}g_*\sG$ in terms of the wild ramification of $\sG$ and the wild ramification of $g$ (precisely, the ramification of the sheaf $\mathrm{R}g_*\Lambda$)?  

In fact, when $\dim_kZ\geq 2$ and $g:Y\rightarrow Z$ a finite radicial cover, the degree of the purely inseparable extension should also be considered in bounding the wild ramification of $\mathrm R g_*\sG$. Here is an example. For an \'etale sheaf on smooth $k$-variety, Saito constructed the characteristic cycle which is a cycle on the cotangent bundle of the variety \cite{cc}. The characteristic cycle is the finest ramification invariant for an \'etale sheaf on a variety as far as we known. 
We take assumptions in Example \ref{Kn}. Let $h:X_n\rightarrow X$ be the canonical projection and $\jmath:U_n\rightarrow X_n$  and $j:U\rightarrow X$ the canonical injections.  Since $h:X_n\rightarrow X$ is finite, surjective and radicial, it gives an equivalence of categories of \'etale sites of $X_n$ and $X$. We have
\begin{equation*}
\mathrm R h_*\jmath_!\sF_n=h_*\jmath_!\sF_n =j_!\sF\ \ \ \textrm{and}\ \ \ \mathrm R h_*\jmath_!\Lambda=h_*\jmath_!\Lambda =j_!\Lambda
\end{equation*}
By \cite[Theorem 7.14]{cc}, for a sheaf of rank $1$ and its ramification is non-degenerate along the ramified divisor, the coefficient of the non-trivial part of the characteristic cycle is contributed by the total dimension of the sheaf along the divisor. 
Hence, by \eqref{dimtotswF}, we have
\begin{align*}
CC(j_!\sF)&=-[\mathbb T^*_XX]-p^{n+1}[D\cdot\langle dx\rangle],\\
CC(j_!\Lambda)&=-[\mathbb T^*_{X_n}X_n]-[D\cdot\langle dz\rangle],\\
CC(\jmath_!\sF_n)&=-[\mathbb T^*_{X_n}X_n]-p[D_n\cdot \langle du\rangle].
\end{align*}
We see that coefficients of $CC(j_!\sF)$ rely on not only those of  $CC(j_!\Lambda)$ and $CC(\jmath_!\sF_n)$ but also the extension degree of $h:X_n\rightarrow X$. This computation can also be derived from \cite[Theorem 4.9]{ccpullback}.
\end{remark}

\section{Geometric perlimilaries}
\subsection{}\label{dilatation}
Let $i:Z\rightarrow X$ be a closed immersion of schemes, $D$ an effective Cartier divisor of $X$ and $X'$ the blow-up of $X$ along $D_Z=D\times_XZ$. We denote by $X^{(D_Z)}$ the complement of the proper transform of $D$ in $X'$ and we call it the {\it dilatation} of $(X,Z)$ along $D$. We have a canonical projection $g:X^{(D_Z)}\rightarrow X$. The pull-back $\wt D$ of $D$ in $X^{(D_Z)}$ is a Cartier divisor of $X^{(D_Z)}$ with the image $D_Z$ in $X$.  If $X=\spec(A)$ is affine,  if $D$ is defined by a non-zero divisor $f$ of $A$ and if $Z$ is defined by an ideal $I$ of $A$, we have 
\begin{equation*}
X^{(D_Z)}=\spec(A[I/f])
\end{equation*}
for $A[I/f]\subseteq A[1/f]$.

Let $h:Y\rightarrow X$ be a morphism of schemes such that $E=D\times_XY$ is an effective Cartier divisor of $Y$ and $i':W\rightarrow Y$ a closed immersion such that $h\circ i':W\rightarrow X$ factors through $i:Z\rightarrow X$. By the universality of the dilatation, $h:Y\rightarrow X$ is uniquely lifted to a morphism $\wt h: Y^{(E_W)}\rightarrow X^{(D_Z)}$.

\subsection{}\label{inflation}
Let $X$ be a scheme, $D$ an effective Cartier divisor of $X$, $M=mD$ a Cartier divisor supported on $D$, $Z_m=\spec(\sO_X[t]/(t^m))$ the $m$-th thickening of the zero section of the line bundle $\bA^1_X$ and $\wt X^{[M]}$ the dilatation of $(\bA^1_X,Z_m)$ along the divisor $\bA^1_D$. We denote by $\wt X^{(M)}$ the complement of the proper transform of $Z_m$ in $\wt X^{(M)}$ and we call it the {\it inflation} of $(X, D)$ of thickening $M$.  If $X=\spec (A)$ is affine and if $D$ is associated to a non-zero divisor $f$ of $A$, we have
\begin{equation*}
\wt X^{(M)}=\spec\left(A[t,s^{\pm 1}]/(f-t^ms)\right)
\end{equation*}

Let $Y$  be a scheme, $E$ an effective Cartier divisor of $Y$ and $h:Y\rightarrow X$ a morphism such that $h^*D=eE$. Let  $N=nE$ be an effective Cartier integer of $Y$ such that $l=en/m$ is an integer and $\wt Y^{(N)}$ the inflation of $(Y,E)$ of thickening $N$. By the universality of the blow-up, the morphism 
\begin{equation*}
 h_l:\bA^1_Y\rightarrow \bA^1_X, (t,y)\mapsto (t^l,h(y))
\end{equation*}
is uniquely lifted to a morphism
\begin{equation}\label{liftinf}
\wt h:\wt Y^{(N)}\rightarrow \wt X^{(M)}.
\end{equation}

\subsection{}\label{bundle}
Let $X$ be a scheme, $R$ a Cartier divisor of $X$, $\mathscr E$ a locally free $\sO_X$-sheaf on $X$ and $E=\spec(\mathrm{Sym}_{\sO_X}(\mathscr E^{\vee}))$ the vector bundle associated to $\mathscr E$. We denote by $E(R)$ the vector bundle $\spec(\mathrm{Sym}_{\sO_X}(\mathscr E\otimes_{\sO_X}\sO_X(R))^{\vee})$.

\subsection{}\label{gsnot}
Let $k$ be a perfect field of characteristic $p>0$, $S$ a connected $k$-scheme, $E$ a vector bundle over $S$ and $G$ a finite \'etale $p$-torsion commutative group scheme over $S$. We denote by $G^{\vee}$ the Cartier dual of $G$ and by $E^{\vee}$ the dual vector bundle of $E$.  Let $\alpha$ be an element in $\Hom_{\gp,S}(E^{\vee}, \bA^1_S)=E^{\vee}(S)$ and $\beta$ an element in $\Hom_{\gp,S}(\bF_p, G)=G(S)$. The pair of group morphisms $(\alpha, \beta)$ give an extension $\wt E$ in $\Ext_S(E,G)$ by following diagram
\begin{equation}\label{pbpo}
\xymatrix{\relax
0\ar[r]&\bF_p\ar[r]&\bA^1_S\ar@{}|-{\Box}[rd]\ar[r]^-(0.5){\rL}&\bA^1_S\ar[r]&0\\
0\ar[r]&\bF_p\ar[r]\ar[d]_-(0.5){\beta}\ar@{=}[u]&E'\ar[r]\ar[d]\ar[u]&E\ar@{=}[d]\ar[r]\ar[u]_-(0.5){\alpha}&0\\
0\ar[r]&G\ar[r]&\wt E\ar[r]&E\ar[r]&0}
\end{equation}
where horizontal lines are exact sequence of group schemes, $\rL:\bA^1_S\rightarrow \bA^1_S$ is the Lang's isogeny $\rL^*(t)=t^p-t$ for the canonical coordinate $t$ of $\bA^1_S$, $E'$ is the pull-back of $\rL:\bA^1_S\rightarrow \bA^1_S$ by $\alpha:E\rightarrow \bA^1_S$ and $\wt E$ is the push-out of $E'$ by  $\beta:\bF_p\rightarrow G$. \'Etale locally on $S$, we have the canonical isomorphism $\Hom_{\gp,S}(G^{\vee}, E^{\vee})\iso E^{\vee}(S)\otimes_{\bF_p}G(S)$. The diagram \eqref{pbpo} gives a morphism of abelian groups
\begin{equation}\label{homtoext}
\Hom_{\gp,S}(G^{\vee}, E^{\vee})\rightarrow \Ext_S(E,G).
\end{equation}
by \'etale descent.
Let $S_n$ be the fiber product $S\times_{k,\F_k^{-n}}k$ where $\F^{-n}_k$ denotes the inverse of the $n$-th power of the Frobenius of $k$. Let $\pi_n:S_n\rightarrow S$ be the composition of the first projection $\pr_1:S_n\rightarrow S$ and the $n$-th power of the Frobenius of $S$. The map $\pr_1:S_n\rightarrow S$ is a $k$-morphism and  gives an equivalence between \'etale sites on $S$ and $S_n$. Hence, the pull-back $\Ext_S(E,G)\rightarrow \Ext_{S_n}(E\times_SS_n, G\times_SS_n)$ is an isomorphism. The map \eqref{homtoext} give rise to a morphism of abelian groups
\begin{equation}\label{injlimgr}
\varinjlim_n\Hom_{\gp, S_n}(G^\vee\times_SS_n, E^\vee\times_SS_n)\rightarrow \Ext_S(E,G).
\end{equation}

\begin{proposition}[{\cite[Proposition 1.20, Lemma 1.21]{wr}}]\label{asgs}
We keep the notation and assumptions of \ref{gsnot}. Then
the map \eqref{injlimgr} is an isomorphism. Let $\varphi:G^{\vee}\rightarrow E^{\vee}$ be an element of the left side of \eqref{injlimgr} and $ E_{\varphi}$ the corresponding extension of $E$ by $G$. Then, $E_{\varphi,\bar s}$ is connected for any geometric point $\bar s$ of $S$, if and only if $\varphi:G^{\vee}\rightarrow E^{\vee}$ is a closed immersion.
\end{proposition}

\section{Ramification of \'etale covers}

\subsection{}
In this section, let $k$ denotes a perfect field of characteristic $p>0$, $X$ a connected, separated and smooth $k$-schemes, $D$ an irreducible effective divisor of $X$ smooth over $\spec(k)$, $i:D\rightarrow X$ and $j:U=X-D\rightarrow X$ the canonical injections, and $\delta:X\rightarrow X\times_kX$ the diagonal map. We denote by $(X\times_kX)'$ the blow-up of $X\times_kX$ along $\delta(D)$ and by $(X\times_kX)^{(D)}$ the complement of the proper transforms $D\times_kX$ and $X\times_kD$ along in $(X\times_kX)'$. We consider $(X\times_kX)^{(D)}$ as an $X$-scheme by the second projection. This projection is smooth (\cite[Lemma 2.1.3]{wr}). By the universality of the blow-ups, the diagonal $\delta:X\rightarrow X\times_kX$ induces a closed immersion $\delta^{(D)}:X\rightarrow (X\times_kX)^{(D)}$.  We denote by $D^{(D)}$ the pull-back of $D\subset X$ in $(X\times_kX)^{(D)}$.  We have the following 
 Cartesian diagram
 \begin{equation}\label{XXD}
\xymatrix{\relax
D\ar[r]^i\ar[d]\ar@{}[rd]|{\Box}&X\ar@{}[rd]|{\Box}\ar[d]^{\delta^{(D)}}&U\ar[d]^{\delta_{U}}\ar[l]_j\\
D^{(D)}\ar@{}[rd]|{\Box}\ar[r]^-(0.5){i^{(D)}}\ar[d]&(X\times_kX)^{(D)\ar[d]^{\pr_2^{(D)}}}\ar@{}[rd]|{\Box}&U\times_kU\ar[l]_-(0.5){j^{(D)}}\ar[d]^{\pr_2}\\
D\ar[r]&X&U\ar[l]}
\end{equation}
where $i^{(D)}$ and $j^{(D)}$ are canonical injections and $\pr^{(D)}_2$ is the second projection. Notice that the composition of $\delta^{(D)}$ and $\pr_2^{(D)}$ is the identity of $X$. We see that the closed subscheme $D^{(D)}$ is an irreducible effective divisor of $(X\times_kX)^{(D)}$ smooth over $\spec(k)$ (\cite[Lemma 2.1.3]{wr}).

\subsection{}\label{sectionXD}
Let $M=mD$ be an effective Cartier divisor of $X$ supported on $D$ and $\wt X^{(M)}$ the inflation of $(X,D)$ of thickening $M$. The scheme $\wt X^{(M)}$ is a smooth $k$-scheme (\cite[Lemma 1.18]{wr}). Let $\beta^{(M)}:\wt X^{(M)}\rightarrow X$ and $\gamma^{(M)}:\wt X^{(M)}\rightarrow \mathbb A^1_k$ be canonical projections. The divisor $\wt D^{(M)}=\gamma^{(M)*}(\{0\})$ is an irreducible effective divisor of $\wt X^{(M)}$ smooth over $\spec (k)$  and we have commutative diagram with Cartesian squares ({\it loc. cit.})
\begin{equation}\label{wtXD}
\xymatrix{
m\wt D^{(M)}\ar[r]^-(0.5){i^{(M)}}\ar[d]\ar@{}[rd]|{\Box}&\wt X^{(M)}\ar@{}[rd]|{\Box}\ar[d]^{\beta^{(M)}}&\mathbb G_m\times_kU\ar[d]^{\pr_2}\ar[l]_-(0.5){j^{(M)}}\\
D\ar[r]_-(0.5){i}&X&U\ar[l]^-(0.5){j}}
\end{equation}
For an effective rational divisor $R=rD$ of $X$ such that $rm$ is an integer, we denote by $\wt R^{(M)}$ the effective divisor $\beta^{(M)*}(R)=rm \wt D^{(M)}$ of $\wt X^{(M)}$.

We denote by  $(X\times_kX)^{(D,M)}$ the inflation of $((X\times_kX)^{(D)}, D^{(D)})$ of thickening $mD^{(D)}$.  Let $\phi^{(M)}:(X\times_kX)^{(D,M)}\rightarrow (X\times_kX)^{(D)}$ and $\psi^{(M)}:(X\times_kX)^{(D,M)}\rightarrow \bA^1_k$ be the canonical projections. Applying \eqref{wtXD} to the smooth $k$-scheme $(X\times_kX)^{(D,M)}$, we see that $(X\times_kX)^{(D,M)}$ and the irreducible effective divisor $D^{(D,M)}=\psi^{(M)*}(\{0\})$ are smooth over $\spec(k)$, that $\phi^{(M)*}(D^{(D)})=mD^{(D,M)}$ and that the complement of $D^{(D,M)}$ in $(X\times_kX)^{(D,M)}$ is $\mathbb G_m\times_kU\times_kU$.

By \eqref{liftinf}, the map $\delta^{(D)}:X\rightarrow (X\times_kX)^{(D)}$ uniquely induces a closed immersion 
\begin{equation*}
\delta^{(D,M)}:\wt X^{(M)} \rightarrow (X\times_kX)^{(D,M)}
\end{equation*}
and the second projection $\pr_2^{(D)}:(X\times_kX)^{(D)}\rightarrow X$ uniquely induces a projection 
\begin{equation*}
\pr_2^{(D,M)}:(X\times_kX)^{(D,M)}\rightarrow \wt X^{(M)}
\end{equation*}
Noticing that $\beta^{(M)}\circ\pr_2^{(D,M)}= \pr^{(D)}_2\circ\phi^{(M)}$ and that $\pr^{(D,M)}\circ\delta^{(D,M)}=\mathrm{id}_{\wt X^{(M)}}$. Combining \eqref{XXD} and \eqref{wtXD}, we have the following Cartesian diagram 
\begin{equation}\label{XXDM}
\xymatrix{\relax
\wt D^{(M)}\ar@{}[rd]|{\Box}\ar[r]^-(0.5){i^{(M)}}\ar[d]&\wt X^{(M)}\ar[d]^{\delta^{(D,M)}}\ar@{}[rd]|{\Box}&\mathbb G_m\times_kU\ar[l]_-(0.5){j^{(M)}}\ar[d]^{\mathrm{id}\times\delta_U}\\
D^{(D,M)}\ar@{}[rd]|{\Box}\ar[r]^-(0.5){i^{(D,M)}}\ar[d]&(X\times_kX)^{(D,M)}\ar[d]\ar[d]^{\pr_2^{(D,M)}}\ar@{}[rd]|{\Box}&\mathbb G_m\times_k (U\times_kU)\ar[d]^{\mathrm{id}\times\pr_2}\ar[l]_-(0.5){j^{(D,M)}}\\
\wt D^{(M)}\ar[r]_-(0.5){i^{(M)}}&\wt X^{(M)}&\mathbb G_m\times_kU\ar[l]^-(0.5){j^{(M)}}}
\end{equation}
where compositions of vertical arrows are identities.

\subsection{}\label{defXXRM}
Let $R=rD$ $(r\in\bQ_{\geq 1})$ be an effective rational divisor of $X$ and $M=mD$ an effective divisor such that $mr$ is an integer. We denote by $(X\times_kX)^{(R,M)}$ the dilatation of $((X\times_kX)^{(D,M)}, \wt X^{(D,M)})$ along the divisor $m(r-1)D^{(D,M)}$, by $\pr^{(R,M)}_2:(X\times_kX)^{(R,M)}\rightarrow \wt X^{(M)}$ the composition of the canonical projections $(X\times_kX)^{(R,M)}\rightarrow (X\times_kX)^{(D,M)}$ and $\pr_2^{(D,M)}:(X\times_kX)^{(D,M)}\rightarrow \wt X^{(M)}$ and by $D^{(R,M)}$ the pull-back of $\wt D^{(M)}\subset \wt X^{(M)}$ in $(X\times_kX)^{(R,M)}$. The scheme $D^{(R,M)}$ is an irreducible divisor of $(X\times_kX)^{(R,M)}$ and both $D^{(R,M)}$ and $(X\times_kX)^{(R,M)}$ are smooth over $\spec(k)$. By the universality of the blow-up, the map $\delta^{(D,M)}:\wt X^{(M)} \rightarrow (X\times_kX)^{(D,M)}$ uniquely lifts to a closed immersion 
\begin{equation*}
\delta^{(R,M)}:\wt X^{(M)} \rightarrow (X\times_kX)^{(R,M)}
\end{equation*}
The diagram \eqref{XXDM} lifts to the following Cartesian diagram
\begin{equation}\label{XXRMtoXM}
\xymatrix{\relax
\wt D^{(M)}\ar@{}[rd]|{\Box}\ar[r]^-(0.5){i^{(M)}}\ar[d]&\wt X^{(M)}\ar[d]^{\delta^{(R,M)}}\ar@{}[rd]|{\Box}&\mathbb G_m\times_kU\ar[l]_-(0.5){j^{(M)}}\ar[d]^{\mathrm{id}\times\delta_U}\\
D^{(R,M)}\ar@{}[rd]|{\Box}\ar[r]^-(0.5){i^{(R,M)}}\ar[d]&(X\times_kX)^{(R,M)}\ar[d]\ar[d]^{\pr_2^{(R,M)}}\ar@{}[rd]|{\Box}&\mathbb G_m\times_k (U\times_kU)\ar[d]^{\mathrm{id}\times\pr_2}\ar[l]_-(0.5){j^{(R,M)}}\\
\wt D^{(M)}\ar[r]_-(0.5){i^{(M)}}&\wt X^{(M)}&\mathbb G_m\times_kU\ar[l]^-(0.5){j^{(M)}}}
\end{equation}
The scheme $D^{(R,M)}$ is a vector bundle over $\wt D^{(M)}$ and we have  (\cite[Corollary 2.9]{wr}), \eqref{bundle}
\begin{equation}\label{DRMiso}
D^{(R,M)}\iso(\mathbb TX\times_X\wt D^{(M)})(-\wt R^{(M)}).
\end{equation}

\begin{proposition}[{\cite[Lemma 2.6]{wr}}]\label{otherdef}
We take the notation and assumptions of \ref{defXXRM}. Let 
\begin{equation*}
\Theta:(X\times_kX)^{(R,M)}\rightarrow \wt X^{(M)}\times_kX
\end{equation*}
be the unique morphism induced by the projection $\pr^{(R,M)}_2:(X\times_kX)^{(R,M)}\rightarrow \wt X^{(M)}$ and the composition $\beta^{(M)}\circ \pr^{(R,M)}_2: (X\times_kX)^{(R,M)}\rightarrow X$ and $\Gamma_{\beta}:\wt X^{(M)}\rightarrow \wt X^{(M)}\times_kX$ the graph of $\beta^{(M)}:\wt X^{(M)}\rightarrow X$. We denote by $ P^{(R,M)}$ the dilatation of $(\wt X^{(M)}\times_kX, \Gamma_{\beta}(\wt X^{(M)}))$ along the effective divisor $\wt R^{(M)}\times_kX$ of $\wt X^{(M)}\times_kX$ (\ref{sectionXD}).  Then, $\Theta:(X\times_kX)^{(R,M)}\rightarrow \wt X^{(M)}\times_kX$ is uniquely lifted to an open immersion $\wt \Theta:(X\times_kX)^{(R,M)}\rightarrow P^{(R,M)}$ and its image is the complement of the proper transform of $\wt X^{(M)}\times_k D\subset \wt X^{(M)}\times_kX$ in $P^{(R,M)}$.
\end{proposition}

\subsection{}
 Let $V$ be an \'etale torsor over $U$ of group $G$, $\Delta:G\rightarrow G\times G$ the diagonal homomorphism,  $W$ the quotient $(V\times_kV)/\Delta(G)$ and $W\rightarrow U\times_kU$ the canonical projection.  The diagonal map $\delta_U:U\rightarrow U\times_kU$ is uniquely lifts to an injection $\varepsilon_U:U\rightarrow W$. Let $R=rD$ $(r\in\bQ_{\geq 1})$ be an effective rational divisor of $X$ and $M=mD$ an effective divisor such that $mr$ is an integer. We denote by $W^{(R,M)}$ the integral closure of $(X\times_kX)^{(R,M)}$ in $\bG_m\times_kW$, by $h:W^{(R,M)}\rightarrow (X\times_kX)^{(R,M)}$ the canonical projection and by $\varepsilon:\wt X^{(M)}\rightarrow W^{(R,M)}$ the unique lifting of $\mathrm{id}\times\varepsilon_U:\bG_m\times_kU\rightarrow \bG_m\times_kW$. We have the following commutative diagram with Cartesian squares
 \begin{equation}\label{WtoUU}
 \xymatrix{\relax
 \bG_m\times_kU\ar@{}[rd]|{\Box}\ar[r]\ar[d]&\bG_m\times_k W\ar[d]\ar[r]\ar@{}[rd]|{\Box}&\bG_m\times_k(U\times_kU)\ar[d]^{j^{(R,M)}}\\
 \wt X^{(M)}\ar[r]_-(0.5){\varepsilon}&W^{(R,M)}\ar[r]\ar[r]_-(0.5){h}&(X\times_kX)^{(R,M)}}
 \end{equation}
 
 \begin{definition}\label{geomram}
 Let $V$ be an \'etale torsor over $U$ of group $G$, $R=rD$ $(r>1)$ an effective rational divisor of $X$ and $M=mD$ an effective divisor such that $mr$ is an integer. We say that {\it the ramification of $V$ over $U$ along $D$ is bounded by $R+$ at a point} $x\in D$ if the finite morphism $h:W^{(R,M)}\rightarrow (X\times_kX)^{(R,M)}$ in \eqref{WtoUU} is \'etale on a Zariski neighborhood of $\varepsilon(\beta^{(M)-1}(x))$ \eqref{wtXD}. We say that {\it the ramification of $V$ over $U$ along $D$ is bounded by $R+$} if the finite morphism $h:W^{(R,M)}\rightarrow (X\times_kX)^{(R,M)}$ is \'etale on a Zariski neighborhood of $\varepsilon(\wt D^{(M)})$. 
 \end{definition}

\begin{remark}
Let $m'$ be an integer divisible by $m$ and $M'=m'D$. We have a surjection $\alpha:\wt X^{(M')}\rightarrow \wt X^{(M)}$ \eqref{liftinf}. By \cite[Lemma 2.13]{wr}, the map $h':W^{(R,M')}\rightarrow (X\times_kX)^{(R,M')}$ associated to $(R,M')$ is \'etale on a Zariski neighborhood of $x'\in \wt X^{(M')}$ if and only if $h:W^{(R,M)}\rightarrow (X\times_kX)^{(R,M)}$ is \'etale on a Zariski neighborhood of the image of $x'$ in $\wt X^{(M)}$. Hence, Definition \ref{geomram} is independent of the choice of the divisor $M$. In the following of this section, we fix $M$ and we omit it in the superscript of all notation for simplicity. The ramification of $V/U$ along $D$ is bounded by $R+$ implies that it is also bounded by $S+$ for an effective divisor $S$ of $X$ supported on $D$ satisfying $S\geq R$ ({\it loc. cit.}).
\end{remark}

\subsection{}\label{Ecirc}
 Let $V$ be an \'etale torsor over $U$ of group $G$. Assume that the ramification of $V/U$ along $D$ is bounded by $R+$, for an effective rational divisor $R=rD$ ($r>1$). Fix an effective divisor $M=mD$  of $X$ such that $mr$ is an integer. We denote by $W^{(R)}_0$ the largest open subscheme of $W^{(R)}$ which is \'etale over $(X\times_kX)^{(R)}$. We define 
$E^{(R)}$ by the following Cartesian diagram
\begin{equation}\label{defER}
\xymatrix{\relax
E^{(R)}\ar[r]\ar@{}[rd]|{\Box}\ar[d]&W_0^{(R)}\ar[d]^h\\
D^{(R)}\ar[r]&(X\times_kX)^{(R)}}
\end{equation}
  By \cite[Corollary 2.15]{wr}, $h|_{E^{(R)}}:E^{(R)}\rightarrow D^{(R)}$ is an \'etale homomorphism of smooth group schemes over $\wt D$. We denote by $E^{(R)\circ}$ the unique connected normal  subgroup scheme of $E^{(R)}$, by $h_0:E^{(R)\circ}\rightarrow D^{(R)}$ the restriction of $h|_{E^{(R)}}:E^{(R)}\rightarrow D^{(R)}$ and by $G^{(R)}$ the kernel of $h_0:E^{(R)\circ}\rightarrow D^{(R)}$. By \cite[Proposition 2.16]{wr}, we have
\begin{itemize}
\item[(i)]
the group scheme $E^{(R)\circ}$ is commutative and $p$-torsion;
\item[(ii)]
 the map $h_0:E^{(R)\circ}\rightarrow D^{(R)}$ is an \'etale and surjective homomorphism of group schemes and $G^{(R)}$ is an \'etale commutative $p$-torsion group scheme over $\wt D$.
\end{itemize}

\begin{definition}\label{nondeg}
We keep the notation and assumptions of \ref{Ecirc}. We say that the ramification of $V/U$ along $D$ bounded by $R+$ is {\it non-degenerate} if the \'etale homomorphism $h_0:E^{(R)\circ}\rightarrow D^{(R)}$ is finite.
\end{definition}

The canonical projection $\wt D\rightarrow D$ is a principal bundle of fiber $\bG_m$. The $\wt D$-morphism $h_0:E^{(R)\circ}\rightarrow D^{(R)}$  is invariant under the canonical $\bG_m$-action. Hence, $h_0$ is finite after replacing $D$ by a Zariski open dense subscheme. Hence, the ramification of $V/U$ along $D$ bounded by $R+$ is non-degenerate after replacing $D$ by a Zariski open dense subscheme.

\begin{definition}\label{geomcharform}
We keep the notation and assumptions of \ref{Ecirc} and we assume that the ramification of $V/U$ along $D$ is non-degenerate. Let $G^{(R)\vee}$ be the Cartier dual of $G^{(R)}$ and $D^{(R)\vee}=(\bT^*X\times_X\wt D)(\wt R)$ the dual vector bundle of $D^{(R)}$. The exact sequence of groups schemes 
\begin{equation}\label{exseqgs}
0\to G^{(R)}\to E^{(R)\circ}\to D^{(R)}\to 0
\end{equation}
corresponds to an injective homomorphism
\begin{equation*}
\ch_R(V/U): G^{(R)\vee}\rightarrow D^{(R)\vee}=(\bT^*X\times_X\wt D)(\wt R)
\end{equation*}
 of group schemes over a finite radicial cover $\wt D_n$ over $\wt D$ (\ref{gsnot} and Proposition \ref{asgs}). We call $\ch_R(V/U)$ the {\it characteristic form of $V/U$ along $D$ with multiplicity $R$}.
\end{definition}

\subsection{}\label{X'toX}
Let $f:X'\rightarrow X$ be a morphism of connected, separated and smooth $k$-schemes, $D'$ an irreducible effective divisor smooth over $\spec(k)$ such that $eD'=f^*D$ and $U'$ the complement of $D'$ in $X'$. We have the following commutative diagram
\begin{equation*}
\xymatrix{\relax
D'\ar[d]\ar[r]^{i'}&X'\ar[d]^f\ar@{}[rd]|{\Box}&U'\ar[l]_{j'}\ar[d]\\
D\ar[r]_i&X&U\ar[l]^j}
\end{equation*}
Let $R=rD$ $(r\in \bQ_{>1})$ be a rational effective divisor of $X$. We fix an effective divisor  $M=mD$ of $X$ such that $m'=m/e$ and $m'r$ are integers and fix an effective divisor $M'=m'D'$ of $X'$. The morphism 
\begin{equation*}
 \mathrm{id}\times f:\bA^1_{X'}\rightarrow \bA^1_X, (t,x')\mapsto (t,f(x'))
\end{equation*}
induces a canonical morphism $\wt f:\wt X'\rightarrow \wt X$ of inflations \eqref{liftinf}. We denote by $R'=f^*R$ the pull-back of $R$ on $X'$. By the left square of \eqref{wtXD}, we have the following Cartesian diagrams
\begin{equation*}
\xymatrix{\relax
\wt D'\ar@{}[rd]|{\Box}\ar[r]\ar[d]&\wt D\ar[d]&&\wt R'\ar@{}[rd]|{\Box}\ar[r]\ar[d]&\wt R\ar[d]\\
\wt X'\ar[r]^-(0.5){\wt f}&\wt X&&\wt X'\ar[r]^-(0.5){\wt f}&\wt X}
\end{equation*}
By the universality of the dilatation and Proposition \ref{otherdef},  $\wt f:\wt X'\rightarrow \wt X$ induces the morphism
\begin{equation*}
\wt f':(X'\times_kX')^{(R')}\rightarrow (X\times_kX)^{(R)}
\end{equation*}
that makes the following diagram commutative
\begin{equation*}
\xymatrix{\relax
(X'\times_kX')^{(R')}\ar[r]^-(0.5){\wt f'}\ar[d]_{\pr^{(R')}_2}&(X\times_kX)^{(R)}\ar[d]^{\pr^{(R)}_2}\\
\wt X'\ar[r]_-(0.5){\wt f}&\wt X}
\end{equation*}
We denote by $D'^{(R')}$ the pull-back of $\wt D'\subset \wt X'$  by $\pr^{(R')}_2$ (cf. \eqref{XXRMtoXM}). 
The diagram above induces the homomorphism $\lambda:D'^{(R')}\rightarrow D^{(R)}\times_{\wt D}\wt D'$ of vector bundles over $\wt D'$ which is functorial to the isomorphism \eqref{DRMiso}, i.e., we have the following commutative diagram (cf. \cite[4.8]{as rc})
\begin{equation*}
\xymatrix{\relax
D'^{(R')}\ar[d]_{\lambda}\ar[r]^-(0.5){\sim}&(\bT X'\times_{X'}\wt D')(-\wt R')\ar[d]^{\partial f}\\
D^{(R)}\times_{\wt D}\wt D'\ar[r]^-(0.5){\sim}&(\bT X\times_{X}\wt D')(-\wt R')}
\end{equation*}

\subsection{} \label{V'}
We keep the notation and assumptions of \ref{X'toX}. Let $V$ be an \'etale torsor of $U$ of group $G$ and we assume that the ramification of $V/U$ along $D$ is bounded by $R+$. Let $V'=U'\times_UV$ be the pull-back \'etale torsor over $U'$, $W'$ the quotient $(V'\times_kV')/\Delta(G)$ and $\varepsilon_{U'}:U'\rightarrow W'$ the unique lifting of the diagonal $\delta_{U'}:U'\rightarrow U'\times_kU'$. We have the following commutative diagram with Cartesian squares
\begin{equation*}
\xymatrix{\relax
\bG_m\times_kW'\ar[r]\ar[d]\ar@{}[rd]|{\Box}&\bG_m\times_k(U'\times_kU')\ar[d]\ar[r]\ar@{}[rd]|{\Box}&(X'\times_kX')^{(R')}\ar[d]^{\wt f'}\\
\bG_m\times_kW\ar[r]&\bG_m\times_k(U\times_kU)\ar[r]&(X\times_kX)^{(R)}}
\end{equation*}
We  further denote by $W'^{(R')}$ the integral closure of $(X'\times_kX')^{(R')}$ in $\bG_m\times_kW'$ and by $\varepsilon':\wt X'\rightarrow W'^{(R')}$ the unique lifting of $\mathrm{id}\times\varepsilon_{U'}:\bG_m\times_k U'\rightarrow \bG_m\times_k W'$. We have the following commutative diagram
\begin{equation*}
\xymatrix{\relax
\wt D'\ar[r]\ar[d]\ar@{}[rd]|{\Box}&\wt X'\ar[d]^{\wt f}\ar[r]^-(0.5){\varepsilon'}&W'^{(R')}\ar[d]\ar[r]^-(0.5){h'}&(X'\times_kX')^{(R')}\ar[d]^{\wt f'}\\
\wt D\ar[r]&\wt X\ar[r]_-(0.5){\varepsilon}&W^{(R)}\ar[r]_-(0.5){h}&(X\times_kX)^{(R)}}
\end{equation*}
We denote by $W'^{(R')}_0$ the largest open subscheme of $W'^{(R')}$ which is \'etale over $(X'\times_kX')^{(R')}$. The canonical morphisms 
\begin{equation*}
W'^{(R')}\to W^{(R)}\times_{(X\times_kX)^{(R)}}(X'\times_kX')^{(R')}\to (X'\times_kX')^{(R')}
\end{equation*}
are finite and surjective. Notice that $W^{(R)}_0\times_{(X\times_kX)^{(R)}}(X'\times_kX')^{(R')}$ is \'etale over the smooth $k$-scheme $(X'\times_kX')^{(R')}$. It is also a smooth $k$-schemes. Hence, by the definition of $W'^{(R')}_0$, we have an open immersion 
\begin{equation}\label{keyopenimm}
W_0^{(R)}\times_{(X\times_kX)^{(R)}}(X'\times_kX')^{(R')}\rightarrow W'^{(R')}_0.
\end{equation}
The inclusion \eqref{keyopenimm} implies that $\varepsilon'(\wt D')$ is contained in $W'^{(R')}_0$ if $\varepsilon(\wt D)$ is contained in $W^{(R)}_0$. Namely, the ramification of $V'/U'$ is bounded by $R'+$ if the ramification of $V/U$ is bounded by $R+$ (Definition \ref{geomram}). 

Assume that the ramification of $V/U$ is bounded by $R+$ and non-degenerate. After replacing $D'$ by a Zariski open dense subset, we may assume that the ramification of $V'/U'$ bounded by $R'+$ is non-degenerate. We denote by $E'^{(R')}$ the pull-back of $D'^{(R')}\subset (X'\times_kX')^{(R')}$ by $h':W'^{(R')}_0\rightarrow (X'\times_kX')^{(R')}$, by $E'^{(R')\circ}$ the unique connected normal subgroup scheme of $E'^{(R')}$, by $h'_0:E'^{(R')\circ}\rightarrow D'^{(R')}$ the restriction of the canonical projection $E'^{(R')}\rightarrow D'^{(R')}$ and by $G'^{(R')}$ the kernel of the homomorphism $h'_0:E'^{(R')\circ}\rightarrow D'^{(R')}$ (cf. \eqref{defER} and \eqref{exseqgs}). Pulling-back the injection \eqref{keyopenimm} by $E'^{(R')}\subset W'^{(R')}_0$, we obtain the following open immersions  \begin{align}\label{ED'toE'}
 E^{(R)}\times_{D^{(R)}} D'^{(R')}\rightarrow E'^{(R')}.
 \end{align}
 Confering \cite[Lemma 1.2, Proposition 1.5]{wr}, the injection \eqref{ED'toE'} is also a homomorphism of group schemes over $\wt D'$. Hence, it induces the following injective homomorphism 
  \begin{align}
E'^{(R')\circ}\rightarrow E^{(R)\circ}\times_{D^{(R)}} D'^{(R')}\label{E'0toE0} 
\end{align}
 and the following injection of kernels of $h_0$ and $h'_0$
 \begin{equation*}
 \iota:G^{(R')}\to G^{(R)}\times_{\wt D}\wt D'.
 \end{equation*}

\begin{proposition}[{cf. \cite[Proposition 2.22]{wr}}]\label{keyfunctorial}
Keep the notation and assumptions of \ref{X'toX} and \ref{V'}.
We have a commutative diagram of group schemes 
\begin{equation}\label{keycommdiag}
\xymatrix{\relax
G^{(R)\vee}\times_{\wt D}\wt D'\ar[rr]^-(0.5){\ch_R(V/U)}\ar[d]_{\iota^{\vee}}&& D^{(R)\vee}\times_{\wt D}\wt D'\ar[d]^{\lambda^{\vee}}\\
G^{(R')\vee}\ar[rr]_{\ch_{R'}(V'/U')}&&D'^{(R')\vee}}
\end{equation}
over a radicial cover $\wt D'_n$ of $\wt D'$ (\ref{asgs}).
\end{proposition}
\begin{proof}
The characteristic form $\ch_R(V/U):G^{(R)\vee}\times_{\wt D}\wt D'\to D^{(R)\vee}\times_{\wt D}\wt D'$ corresponds to the exact sequence of group schemes
\begin{equation}\label{bcbase}
0\to G^{(R)}\times_{\wt D}\wt D'\to E^{(R)\circ}\times_{\wt D}\wt D'\to D^{(R)}\times_{\wt D}\wt D'\to 0,
\end{equation}
Then, the composition of $\lambda^{\vee}\circ\ch_R(V/U)$ associates to the pull-back of \eqref{bcbase} by $\lambda$, i.e., the following exact sequence
\begin{equation}\label{bottomGR}
0\to G^{(R)}\times_{\wt D}\wt D'\to E^{(R)\circ}\times_{D^{(R)}}D'^{(R')}\to D'^{(R')}\to 0.
\end{equation}
The characteristic form $\ch_{R'}(V'/U'):G'^{(R')\vee}\rightarrow D'^{(R')\vee}$ corresponds to the exact sequence
\begin{equation}\label{exactG'}
0\to G^{(R')}\to E'^{(R')\circ}\to D'^{(R')}\to 0
\end{equation}
The composition of $\iota^{\vee}\circ\ch_{R'}(V'/U')$ associates to the push out of the exact sequence \eqref{exactG'} by $\iota:G^{(R')}\to G^{(R)}\times_{\wt D}\wt D'$. The injection \eqref{E'0toE0} induces the following commutative diagram of group schemes over $\wt D'$
\begin{equation}\label{E'0E0}
\xymatrix{\relax
0\ar[r]&G^{(R')}\ar@{}|-{\Box}[rd]\ar[r]\ar[d]_{\iota}&E'^{(R')\circ}\ar[r]\ar[d]&D'^{(R')}\ar@{=}[d]\ar[r]&0\\
0\ar[r]&G^{(R)}\times_{\wt D}\wt D'\ar[r]&E^{(R)\circ}\times_{D^{(R)}}D'^{(R')}\ar[r]&D'^{(R')}\ar[r]&0}
\end{equation}
where horizontal lines are exact sequences and the left square is Cartesian. Notice that $E^{(R)\circ}\times_{D^{(R)}}D'^{(R')}$ is the push-out of $E'^{(R')\circ}$ and $G^{(R)}\times_{\wt D}\wt D'$ over $G^{(R')}$. Hence, $\lambda^{\vee}\circ\ch_R(V/U)$ and $\iota^{\vee}\circ\ch_{R'}(V'/U')$ correspond to the same exact sequence  \eqref{bottomGR}. Hence we obtain the commutativity of the diagram \eqref{keycommdiag}.
 \end{proof}

\section{Functoriality of characteristic forms}

In this section, let $k$ be a perfect field of characteristic $p>0$ and $K$ a discrete valuation field which is a $k$-algebra with $\sO_K$ henselian. 

\begin{definition}\label{geomreal}

We say that $K$ is {\it geometric} if there is a triple $(X,D,\xi)$, where $X$ is a connected, separated and smooth $k$-scheme and $D$ an irreducible Cartier divisor of $X$ smooth over $\spec(k)$ with the generic point $\xi$, such that $\sO_{X,\xi}^{\rh}\iso\sO_K$. Such a triple $(X,D,\xi)$ is called a {\it geometric realization} of $K$.  In this case, the residue field $F=\kappa(\xi)$ of $\sO_K$  is of finite type over $k$ and $\sO_K$ is a finite $\sO_K^p$-algebra. The absolute differential module $\Omega^1_{\sO_K}=\Omega^1_{\sO_K/k}=\Omega^1_{\sO_K/\sO_K^p}$ is a finite generated $\sO_K$-module. 

We say that an extension $K'/K$ of geometric discrete valuation fields is {\it geometric} if there exist geometric realizations $(X,D, \xi)$ and $(X',D', \xi')$ of $K$ and $K'$, and a dominant $k$-morphism $g:X'\rightarrow X$ such that $g(\xi)=\xi'$, that $D'=(D\times_XX')_{\mathrm{red}}$ and that the diagram
\begin{equation*}
\xymatrix{\relax
\sO_{X,\xi}^{\rh}\ar[r]^-(0.5){\sim}\ar[d]_{g^*}&\sO_K\ar[d]\\
\sO_{X',\xi'}^{\rh}\ar[r]^-(0.5){\sim}&\sO_{K'}}
\end{equation*}
is commutative.
\end{definition}

\begin{proposition}[{\cite[Proposition 2.27, Proposition 2.28]{wr}}]\label{geomlocal}
Let $K$ be a geometric discrete valuation field, $L$ a finite Galois  $K$-algebra of group $G$
, $(X,D,\xi)$ a geometric realization of $K$ and $V$ an \'etale $G$-torsor of $U=X-D$ such that $\spec(L)=\spec(K)\times_XV$. 
\begin{itemize}
\item[(1)]
Then, for a rational number $r>1$, the ramification of $L/K$ is bounded by $r+$ (Definition \ref{rambound}) if and only if, the ramification of $V/U$ along $D$ is bounded by $R+$ $(R=rD)$ at $\xi\in D$ (Definition \ref{geomram}).
\item[(2)] Assume that the ramification of $L/K$ has conductor $r>1$, i.e., $G^r\neq \{1\}$ and $G^{r+}=\{1\}$.  After replacing $X$ by a Zariski open dense neighborhood of $\xi$, we may assume that the ramification of $V/U$ along $D$ is bounded by $R+$.
Then, we have
$G^{r+}=G^{(R)}_{\bar \xi}$, where $G^{(R)}$ denotes the \'etale group scheme over $\wt D$ constructed in \ref{Ecirc} and $\bar\xi$ a geometric point above $\xi\in D$ factors through $\wt D$. In particular, $G^r$ and $G^r_K/G^{r+}_K$ are finite and abelian $p$-torsion groups. 
\end{itemize}
\end{proposition}

\begin{definition}
We keep the assumption of Proposition \ref{geomlocal} (2).
The stalk of the injective homomorphism
$\ch_R(V/U):G^{(R)\vee}\rightarrow D^{(R)\vee}=(\bT^*X\times_X\wt D)(\wt R)$ (\ref{geomcharform})
 at a geometric point $\bar\xi$ above $\xi\in D$ gives the following injective homomorphism, called the {\it characteristic form of $L/K$}
\begin{equation}\label{finicharform}
\ch(L/K):\Hom_{\bF_p}(G^r,\bF_p)\rightarrow \Omega^1_{\sO_K}\otimes_{\sO_K}\overline N_K^{-r},
\end{equation}
where $\overline N^{-r}$ denotes the $1$-dimensional $\overline F$-vector space
\begin{equation*}
\overline N_K^{-r}=\{x\in K^{\sep}\,;\, \bv_K(x)\geq -r\}/\{x\in K^{\sep}\,;\, \bv_K(x)> -r\}.
\end{equation*}
 Pass \eqref{finicharform} to the absolute Galois group of $K$. Then, for any rational number $r>1$, we have the following injective homomorphism, called the {\it characteristic form}
\begin{equation*}
\ch:\Hom_{\bF_p}(G^r_K/G^{r+}_K,\bF_p)\rightarrow \Omega^1_{\sO_K}\otimes_{\sO_K}\overline N_K^{-r}.
\end{equation*}
\end{definition}

\begin{theorem}\label{themgr}
Let $K'/K$ be a geometric extension of geometric discrete valuation fields, $e$ the ramification index of $K'/K$ and $\gamma:G_{K'}\rightarrow G_K$ the induced morphism of absolute Galois group. Then, for any  $r\in \mathbb Q_{>1}$,
we have $\gamma(G_{K'}^{er})\subset G^r_K$ and the following commutative diagram
\begin{equation}\label{funcharform}
\xymatrix{\relax
\Hom_{\bF_p}(G^r_K/G^{r+}_K,\bF_p)\ar[rr]^-(0.5){\ch}\ar[d]_{\gamma^{\vee}}&&\Omega^1_{\sO_K}\otimes_{\sO_K}\overline N_K^{-r}\ar[d]^{\theta}\\
\Hom_{\bF_p}(G^{er}_{K'}/G^{er+}_{K'},\bF_p)\ar[rr]^-(0.5){\ch}&&\Omega^1_{\sO_{K'}}\otimes_{\sO_{K'}}\overline N_{K'}^{-er}}
\end{equation}
where $\theta$ denotes the canonical map of differentials induced by the inclusion $\sO_K\subset \sO_{K'}$.
\end{theorem}
\begin{proof}
Let $L$ be a finite Galois extension of $K$ of group $G$, $L'$ a Galois extension of $K'$ which is a component of $L\otimes_KK'$ and $H=\gal(L'/K')$.

To verify $\gamma(G_{K'}^{er})\subset G^r_K$ for any $r\in \bQ_{>1}$, it sufficient to show that $H^{er}=\{1\}$ if $G^r=\{1\}$. Let $(X,D,\xi)$ and $(X',D',\xi')$ be geometric realizations of $K$ and $K'$ and $g:X'\rightarrow X$ the map as in Definition \ref{geomreal}. After replacing $X$ by a Zariski neighborhood of $\xi$, we can find an \'etale $G$ torsor $V$ of $U=X-D$ such that $\spec(L)=V\times_U\spec(K)$. We put $V'=V\times_XX'$.  By the construction in \ref{V'}, we see that the ramification of $V'/U'$ along $D'$ is bounded by $esD+$ for an $s\in \bQ_{>1}$ if the ramification of $V/U$ along $D$ is bounded $sD+$. Hence, for any $s\in \bQ_{>1}$, we have $H^{es+}=\{1\}$ if $G^{s+}=\{1\}$ (Proposition \ref{geomlocal}). Then, we have
\begin{align*}
G^r=\{1\}&\Rightarrow G^{(r-\epsilon)+}=\{1\},\ \ \ \textrm{for some}\ \  \epsilon>0\,;\\
&\Rightarrow H^{(er-e\epsilon)+}=\{1\}\Rightarrow H^{er}=\{1\}.
\end{align*}

To verify the commutativity of the diagram \eqref{funcharform}, we only need to check it for a finite quotient of $G_K$.  
We replace $G_K$ by $G$ and $G_{K'}$ by $H$. We further replace $L$ by the sub-field $L^{G^{r+}}$. We have $G^{r+}=\{1\}$ and hence $H^{er+}=\{1\}$.  We are reduced to show the commutativity of the diagram
\begin{equation}\label{finitefuncharform}
\xymatrix{\relax
\Hom_{\bF_p}(G^r,\bF_p)\ar[rr]^-(0.5){\ch(L/K)}\ar[d]_{\gamma^{\vee}}&&\Omega^1_{\sO_K}\otimes_{\sO_K}\overline N_K^{-r}\ar[d]^{\theta}\\
\Hom_{\bF_p}(H^{er},\bF_p)\ar[rr]^-(0.5){\ch(L'/K')}&&\Omega^1_{\sO_{K'}}\otimes_{\sO_{K'}}\overline N_{K'}^{-er}}
\end{equation}
We take the geometric realizations of $K$, $K'$ and $L$ as above and we put $R=rD$ and ($R'=erD'$). Notice that $g:X'\rightarrow X$ is dominant. We may shrinking $X$ (resp. $X'$) by a Zariski neighborhood of $\xi$ (resp. $\xi'$) such that the ramification of $V/U$ along $D$ is bounded by $R+$ and non-degenerate (resp. the ramification of $V'/U'$ along $D'$ is bounded by $R'+$ and non-degenerate). Taking the notation in \ref{X'toX} and \ref{V'},  we have the commutative diagram \eqref{keycommdiag} of group schemes over a radicial cover $\wt D'_n$ of $\wt D'$. Let $\bar \xi'$ be a geometric point above $\xi'\in D'$ that factors through $\wt D'_n$.  Taking the stalk of the diagram \eqref{keycommdiag} at $\bar \xi'$, we obtain \eqref{finitefuncharform}.
\end{proof}

\begin{proposition}\label{coropurins}
Let $K'/K$ be a geometric extension of geometric discrete valuation field. We assume that $K'/K$ is purely inseparable, untwisted and has finite exponent (\ref{erd}). We denote by $e$ the ramification index of $K'/K$, by $f$ the dual ramification index of $K'/K$ (\ref{eri}) and by $\gamma:G_{K'}\rightarrow G_K$ the induced isomorphism of absolute Galois groups. Notice $K'/K$ has exponent $n$ for $ef=p^n$.
Then, for any $r\in \bQ_{>1}$, we have the following commutative diagrams
\begin{equation*}
\xymatrix{\relax
\Hom_{\bF_p}(G^r_K/G^{r+}_K,\bF_p)\ar[r]^-(0.5){\ch}\ar[d]_{\gamma^{\vee}}&\Omega^1_{\sO_K}\otimes_{\sO_K}\overline N_K^{-r}\ar[d]^{\theta}\\
\Hom_{\bF_p}(G^{er}_{K'}/G^{er+}_{K'},\bF_p)\ar[r]^-(0.5){\ch}&\Omega^1_{\sO_{K'}}\otimes_{\sO_{K'}}\overline N_{K'}^{-er}}
\end{equation*}
\begin{equation*}
\xymatrix{\relax
\Hom_{\bF_p}(G^r_{K'}/G^{r+}_{K'},\bF_p)\ar[r]^-(0.5){\ch}\ar[d]_{(\gamma^{-1})^{\vee}}&\Omega^1_{\sO_{K'}}\otimes_{\sO_{K'}}\overline N_{K'}^{-r}\ar[d]^{\sigma}\\
\Hom_{\bF_p}(G^{fr}_{K}/G^{fr+}_{K},\bF_p)\ar[r]^-(0.5){\ch}&\Omega^1_{\sO_{K}}\otimes_{\sO_{K}}\overline N_{K}^{-fr}}
\end{equation*}
where $\sigma$ denotes the canonical map induced by the inclusion $\iota:K'\rightarrow K$, $x'\mapsto x^{p^n}$.
\end{proposition}
\begin{proof}
Applying Theorem \ref{themgr} to the extension $K'/K$, we obtain the commutativity of the first diagram. To verify the commutativity of the second diagram by Theorem \ref{themgr}, it is sufficient to show that the injection $\iota:K'\rightarrow K$ is geometric. Let 
$(X,D,\xi)$ be a geometric realization of $K$ and $\F_X^n:X\rightarrow X$ the $n$-th power of the absolute Frobenius of $X$. Since the purely inseparable extension $K'/K$ has exponent $n$, we may find a realization $(X',D',\xi')$ of $K'$ and a $k$-morphism $g:X'\rightarrow X$ such that $F_X^n:X\rightarrow X$ is the composition of $h:X\rightarrow X'$ and $g:X'\rightarrow X$  after shrinking $X$. Let $\F_k^{-n}$ be the inverse of the $n$-th power of the Frobenius of $k$ and $X_n=X\times_{k,\F_k^{-n}}k$. The first projection $\pr_1:X_n\rightarrow X$ is an isomorphism of schemes and the composition $h\circ\pr_1:X_n\rightarrow X'$ is a $k$-morphism. Hence, $\iota:K'\rightarrow K$ is geometric and realized by  $h\circ\pr_1:X_n\rightarrow X'$. 
\end{proof}

\begin{conjecture}
Keep the assumptions of Corollary \ref{coropurins}. Then, two commutative diagrams in {\it loc. cit.} are co-Cartesian. In particular, assuming that $K'/K$ has exponent $1$, for any $r\in \bQ_{>1}$, the sequence 
\begin{equation*}
G_K^{pr}/G^{pr+}_K\xrightarrow{\gamma^{-1}}G_{K'}^{er}/G^{er+}_{K'}\xrightarrow{\gamma} G_K^{r}/G^{r+}_K
\end{equation*}
is exact.
\end{conjecture}

\subsection*{Acknowledgement}
The author is grateful to T. Saito for many useful comments and  thank Y. Cao and J.-P. Teyssier for discussions. The author is a postdoctoral fellow at the Max-Planck Institute of Mathematics in Bonn and he would like to thank the institute for the hospitality.

\end{document}